\documentclass{article}

\usepackage{arxiv}

\usepackage[utf8]{inputenc} 
\usepackage[T1]{fontenc}    
\usepackage{hyperref}
\usepackage{url}            
\usepackage{booktabs}       
\usepackage{amsfonts}       
\usepackage{nicefrac}       
\usepackage{microtype}      
\usepackage{xcolor}         
\usepackage{amsmath,amsfonts,amssymb,amsthm,array}
\usepackage[font=small]{caption}
\usepackage{subcaption}
\usepackage{bm}
\usepackage{color}
\usepackage{graphicx}
\usepackage{enumitem}
\usepackage{bbm}
\usepackage{threeparttable}
\usepackage{algorithm}
\usepackage{algpseudocode}
\usepackage{wrapfig}

\newtheorem{lemma}{Lemma}
\newtheorem{theorem}{Theorem}
\newtheorem{corollary}{Corollary}
\newtheorem{proposition}{Proposition}
\newtheorem{assumption}{Assumption}

\graphicspath{ {./Figures/} }

\usepackage{amsfonts}
\usepackage{amsmath}
\usepackage{amsthm}
\usepackage{amssymb}

\usepackage{xcolor}
\usepackage{colortbl}
\usepackage{color}
\usepackage{graphicx}
\usepackage{multirow}
\usepackage{pifont}

\usepackage{mathtools}

\usepackage{verbatim}
\usepackage{xspace} %

\usepackage{enumitem}

\newcommand{\dotprod}[2]{\left\langle #1,#2 \right\rangle}
\newcommand{\norms}[1]{\left\| #1 \right\|}
\newcommand{\expect}[1]{\mathbb{E}\left[ #1 \right]}

\definecolor{mydarkgreen}{RGB}{39,130,67}
\definecolor{mydarkred}{RGB}{192,47,25}

\definecolor{bgcolor2}{rgb}{0.8,1,1}
\definecolor{bgcolor}{rgb}{0.8,1,0.8}

\providecommand{\norm}[1]{\left\lVert#1\right\rVert}

  \providecommand{\R}{\mathbb{R}} %
  \providecommand{\N}{\mathbb{N}} %
  
  \DeclareMathOperator{\E}{{\mathbb E}}
  \providecommand{\EEb}[2]{{\mathbb E}_{#1}\left[#2\right] } %

  \DeclareMathOperator*{\argmin}{arg\,min}

  \providecommand{\mM}{\mathbf{M}}
  
  \providecommand{\mO}{\mathbf{O}}

  \providecommand{\cO}{\mathcal{O}}

  \providecommand{\cX}{\mathcal{X}}

  \providecommand{\cZ}{\mathcal{Z}}

  \usepackage{bm}

\newcommand{\circledOne}{\text{\ding{172}}}
\newcommand{\circledTwo}{\text{\ding{173}}}
\newcommand{\circledThree}{\text{\ding{174}}}
\newcommand{\circledFour}{\text{\ding{175}}}

\usepackage{url}

\newcommand{\normp}[1]{\left\|#1\right\|_p}

\def\<#1,#2>{\langle #1,#2\rangle}

\usepackage[textwidth=5cm]{todonotes}
  
\providecommand{\mycomment}[3]{\todo[inline, caption={},size=footnotesize,color=#1!20]{\textbf{#2: }#3}}%
\newcommand\commenter[2]%
{%
  \expandafter\newcommand\csname #1\endcsname[1]{\mycomment{#2}{#1}{##1}}
}

\commenter{Andrew}{green} 

\title{Methods for Optimization Problems with Markovian Stochasticity and Non-Euclidean Geometry}

\author{%
  Vladimir Solodkin\\
  Basic Research of Artificial Intelligence Laboratory (BRAIn Lab)
  \And
  Andrew Veprikov\\
  Basic Research of Artificial Intelligence Laboratory (BRAIn Lab)
  \And
  Aleksandr Beznosikov \\
  Basic Research of Artificial Intelligence Laboratory (BRAIn Lab)}

\begin{document}

\maketitle

\begin{abstract}
    This paper examines a variety of classical optimization problems, including well-known minimization tasks and more general variational inequalities. We consider a stochastic formulation of these problems, and unlike most previous work, we take into account the complex Markov nature of the noise. We also consider the geometry of the problem in an arbitrary non-Euclidean setting, and propose four methods based on the Mirror Descent iteration technique. Theoretical analysis is provided for smooth and convex minimization problems and variational inequalities with Lipschitz and monotone operators. The convergence guarantees obtained are optimal for first-order stochastic methods, as evidenced by the lower bound estimates provided in this paper.
\end{abstract}
\addtocontents{toc}{\protect\setcounter{tocdepth}{0}}
\vspace{-0.1cm}
\section{Introduction}
\label{section:introduction}
\vspace{-0.1cm}
    In the quest to solve complex real-world problems, optimization plays a crucial role in various fields, including artificial intelligence \cite{creswell2018generative, zhang2018improved, hashem2024adaptive}, finance \cite{cornuejols2006optimization}, operations research \cite{rardin1998optimization}, and reinforcement learning \cite{sutton2018reinforcement, mnih2015human}. While traditional deterministic optimization methods assume that all problem data is known and fixed, practical scenarios often involve uncertainty and variability. Stochastic optimization is providing to be a powerful method \cite{robbins1951stochastic} to address these challenges. This approach incorporates randomness into the optimization process, allowing it to handle the unpredictable nature of application problems more effectively. This enables the generation of more robust and reliable solutions. The sources of stochastic behavior can include noise \cite{huang2021stochastic}, sampling methods \cite{hedar2020estimation}, and environmental factors \cite{gorbunov2020stochastic}. To describe this randomness, one of the simplest and most common ways is to assume that noise variables are independent and identically distributed (i.i.d.) \cite{zhou2017stochastic, johnson2013accelerating, nazykov2024stochastic}. However, in modern applications, there is an increasing number of problems where the noise depends on a certain background. Therefore, the assumption of independence is not always satisfied. For instance, in reinforcement learning \cite{bhandari2018finite, srikant2019finite, durmus2021stability} or in distributed optimization \cite{lopes2007incremental, dimakis2010gossip, even2023stochastic}. Thus, we naturally arrive at the statement of the problem in a more general form, where the noise variables are the realizations of an ergodic Markov chain. 
    
    In the first paper on Markovian optimization \cite{duchi2012ergodic}, the authors immerse this type of stochasticity in arbitrary geometry, but they consider only the non-smooth problems. Recently, for the general case of Markovian noise the finite-time analysis of non-accelerated SGD-type algorithms has been studied in \cite{sun2018markov} in the convex case and in \cite{doan2022finite} in the non-convex and strongly convex settings. Alongside this, \cite{dorfman2022adapting} proposed a random batch size algorithm that adapts to the mixing time of the underlying chain for non-convex optimization with a compact domain. In the exploration of accelerated SGD, the paper \cite{doan2020convergence} obtained estimates and the article \cite{beznosikov2024first} improved these results for both non-convex and strongly convex settings. At the same time, numerous papers have appeared dealing with the special scenario of distributed optimization \cite{even2023stochastic, doan2022finite}. Unfortunately, all of these works only consider the Euclidean formulation, while it is interesting to investigate the arbitrary geometry setting as \cite{duchi2012ergodic}, but in the smooth case, since it is possible to design accelerated methods in it \cite{nesterov1983method}.
    
    The classical algorithm that makes use of non-Euclidean properties is Mirror Descent (MD) \cite{nemirovsky1983wiley}. Various modifications of this method have been extensively studied, including Stochastic MD \cite{zhou2017stochastic, jin2020efficiently}, Composite MD \cite{duchi2010composite, lei2018stochastic}, Accelerated MD \cite{lan2012optimal, krichene2015accelerated, lan2019unified}, Coordinate MD \cite{gao2020randomized, hanzely2021fastest}, Coupled MD \cite{allen2014linear} and even Zero-order MD \cite{gorbunov2022accelerated}. Recently, the first work combining MD with the Markovian noise has appeared \cite{xiao2024noise}. However, this paper only considers a token-algorithm (finite-sum stochasticity) in the special setting of decentralized optimization, and also proposes a non-accelerated method.

    When considering the generalisation of optimisation problems, it is possible to do this not only from the perspective of randomness or geometry, but also from the overall formulation of the problem to be solved. For instance, we can consider a wide class of variational inequalities (VI). VI cover important special cases, including minimization over a convex domain, saddle point/min-max, and fixed point problems with a wide variety of applications in supervised \cite{joachims2005support,bach2012optimization} and unsupervised \cite{xu2004maximum,bach2008convex} learning, image denosing \cite{chambolle2011first}, computational game theory \cite{facchinei2003finite}, and GANs \cite{gidel2018variational}. As demonstrated  in \cite{goodfellow2016nips}, the classical Gradient/Mirror Descent algorithms can diverge for the VI problem, therefore, more sophisticated techniques need to be introduced. One of the first algorithms dealing with VI is Extragradient method \cite{korpelevich1976extragradient}, followed by Dual Extragradient \cite{Nesterov2003DualEA} and Mirror-Prox (MP) \cite{nemirovski2004prox}. As shown in \cite{nemirovski2004prox}, the Mirror-Prox algorithm attains an optimal convergence rate in monotone variational inequalities with Lipschitz continuous operators. The pioneer work dealing with MP in the stochastic setting was done in \cite{juditsky2011solving}. Later, the analysis of stochastic finite-sum variational inequalities or saddle point problems has been extensively studied by many authors, but mostly in the i.i.d. setting \cite{chavdarova2019reducing, alacaoglu2022stochastic, beznosikov2023stochastic}. In the Markovian noise setup, \cite{wang2022stability} considered the finite sum case in the narrowing of the saddle point problem only. Whereas, \cite{beznosikov2024first} examines general formulation of VI, yet again in the Euclidean setup. 

    Motivated by these research gaps, we arrived at the following results.
    \vspace{-3mm}
    \subsection{Our contribution}
    \vspace{-2mm}
    Our main contributions are:
    
    $\bullet$ We present new algorithms based on Mirror Descent and Mirror-Prox methods for minimization problems and for variational inequalities, respectively. The analysis is provided in the general setup of arbitrary norms and compatible Bregman divergences, which is quite rare for the Markovian case.

    $\bullet$ We propose two possible techniques for calculating the gradient estimator: with batching (Algorithms \ref{alg:MDG} and \ref{alg:MPG}), and without batching (Algorithms \ref{alg:MDGwb} and \ref{alg:MPGwb}).
    This paper also provides lower bounds for minimization and VI problems with Markovian noise (Proposition \ref{proposition:1} and \ref{proposition:2}), allowing us to obtain optimal convergence rates for Algorithms \ref{alg:MDG} and \ref{alg:MPG} using batches of size only $\widetilde{\mathcal{O}}(1)$ (Theorem \ref{theorem:MDG} and Theorem \ref{theorem:MPG}).
    
    $\bullet$ For the non-batched version of the algorithm that solves the VI problem (Algorithm \ref{alg:MPGwb}), we manage to show the convergence under the assumption that the variance is bounded only in expectation with respect to the stationary distribution (Theorem \ref{theorem:MPGwb}), which has not been achieved with Markovian noise before.
    
    $\bullet$ As a byproduct of our main results, we provide a novel deviation bound for the mean of realizations of geometrically ergodic Markov chains (Lemma \ref{lem:xuivjopeentertainment}) in an arbitrary norm. To the best of our knowledge, this result has only been obtained in the Euclidean setup \cite{beznosikov2024first, paulin2015concentration}.
    
    \vspace{0.2cm}
\textbf{Notations. }
We use $\dotprod{x}{y} := \sum_{j = 1}^d x_i y_i$ to denote standard inner product of vectors $x, y \in \mathbb{R}^d$. We introduce $l_p$-norm of vector $x \in \mathbb{R}^d$ as $\normp{x}^p := \sum_{j=1}^d |x_j|^p$. We use $\|\cdot\|_*$ to denote the norm conjugate to $\|\cdot\|$, in particular $\|\cdot\|_* := \max_{z:\|z\|\leq1}\langle \cdot, z\rangle$. We define $\textit{poly}(x)$ to be a notation of polynomial dependence, that is, $\textit{ poly}(x) := x^k$ for some $k\geq0$. We use classical notation that $f(t) = \cO( g(t))$ if exists $C > 0$ such that $f(t) \leq C g(t)$ for all $t \in \N$ and we use $f(t) = \widetilde{\cO} (g(t))$, if exists $C > 0$ such that $f(t) \leq C \log(t) g(t)$ for all $t \in \N$.
Let $(\mathsf{Z},\mathsf{d}_{\mathsf{Z}})$ be a complete separable metric space endowed with its Borel $\sigma$-field $\mathcal{Z}$. We denote by $(\mathsf{Z}^{\ensuremath{\mathbb{N}}}, \mathcal{Z}^{\otimes \ensuremath{\mathbb{N}}})$ the corresponding canonical process. Consider the Markov kernel defined on $\mathsf{Z} \times \mathcal{Z}$, and denote by $\mathbb{P}_{\xi}$ and $\E_{\xi}$ the corresponding probability distribution and the expected value with initial distribution $\xi$. Let $(Z_k)_{k \in \ensuremath{\mathbb{N}}}$ be the corresponding canonical process. For $\xi = \delta_{z}$, $z \in \mathsf{Z}$, we simply write $\mathbb{P}_{z}$ and $\E_{z}$ instead of $\mathbb{P}_{\delta_{z}}$ and $\E_{\delta_{z}}$. For $x^{0},\ldots,x^{t}$ being the iterates of any algorithm, we denote $\mathcal{F}_{t} = \sigma(x^{j}, j \leq t)$ and write $\E_{t}$ to denote $\E[\cdot | \mathcal{F}_{t}]$.

\section{Main results}
\vspace{-2mm}
\subsection{Markovian Accelerated Mirror Descent}
\label{subsection:MAMD}

In this section, we study the optimization problem of the form
\begin{equation}
    \label{eq:problem1}
    f^* := \min_{x\in\mathcal{X}} \left\{ f(x) := \E_{Z\sim\pi}[F(x, Z)] \right\},
\end{equation}
where $\pi$ is a usually unknown distribution and $\mathcal{X}$ is a normed vector space with a dual space $\mathcal{X}^*$ and pair of primal and dual norms $\| \cdot \|$, $\| \cdot \|_*$. 
Let $\omega(\cdot)$ be a proper convex lower semicontinuous and $1$-strongly convex function with respect to $\| \cdot \|$ on $\cX$. Then for any $x \in \cX$, for which $\partial \omega(x)$ is non-empty and for any $y \in \cX$ we can define the Bregman divergence as 
\begin{equation*}
    V(x, y) := \omega(y) - \omega(x) - \dotprod{\omega'(x)}{y - x}, ~~\omega'(x) \in \partial \omega(x) .
\end{equation*}
We assume that for all optimization problems, that are considered in this paper, there exists the Bregman divergence $V$ with respect to an arbitrary norm $\| \cdot \|$ on $\cX$. We also assume that $\cX$ is compact, and $D^2 := \max_{x, y \in \cX}V(x, y)$. From 1-strongly convexity of $\omega$ it follows that $\norm{x - y}^2 \leq 2 D^2$ for all $x, y \in \cX$. 
For all $\xi \in \R^d$ we also define the \textit{prox mapping} $P_x(\xi)$ as 
\begin{equation*}
    P_x(\xi) := \arg\min_{y \in \cX} \left\{ V(x, y) + \dotprod{\xi}{y} \right\}.
\end{equation*}
We now introduce the common assumptions, required for the analysis of solving \eqref{eq:problem1}.
\captionsetup{type=assumption}
\renewcommand{\theassumption}{1}
\begin{assumption}
    \label{as:lip}
    The function $f$ is $L$-smooth on $\mathcal{X}$ with respect to $\| \cdot \|$ norm, i.e., it is differentiable and there exists $L > 0$ such that for any $x, y \in \mathcal{X}$ the following inequality holds
    \begin{equation*}
        \| \nabla f(x) - \nabla f(y)\|_* \leq L\| x - y\|.
    \end{equation*}
\end{assumption}
\captionsetup{type=assumption}
\renewcommand{\theassumption}{2}
\begin{assumption}
    \label{as:conv}
    The function $f$ is convex on $\mathcal{X}$, i.e., it is differentiable and for any $x, y \in \mathcal{X}$ the following inequality holds 
    \begin{equation*}
        f(y) \leq f(x) - \langle \nabla f(y), x - y\rangle.
    \end{equation*}
\end{assumption}
We assume that the access to the function $f$ and its gradient is only available through the noisy oracles $F(x, Z)$ and $\nabla F (x, Z )$, respectively. Further, we make assumptions on the noise variables $\left\{ Z_t \right\}_{t = 0}^{+\infty}$. 
\captionsetup{type=assumption}
\renewcommand{\theassumption}{3}
\begin{assumption}
    \label{as:noise}
    $\{Z_t\}_{t=0}^{\infty}$ is a stationary Markov chain on $(\mathsf{Z},\mathcal{Z})$ with unique invariant distribution $\pi$. Moreover, $\{Z_t\}_{t=0}^{\infty}$ is uniformly geometrically ergodic with mixing time $\tau_{\text{mix}}$, i.e. for all $t > 0$, $z_0, z \in \cZ$ the following inequality holds
    \begin{equation*}
        \left| \mathbb{P} \left\{ Z_t = z | Z_0 = z_0 \right\} - \pi_z \right| \lesssim (1/2)^{t / \tau_{\text{mix}}} .
    \end{equation*}
\end{assumption}
Assumption \ref{as:noise} is a common occurrence in the literature on Markovian noise \cite{creswell2018generative, doan2020convergence, dorfman2022adapting, beznosikov2024first}. This assumption covers finite Markov chains with irreducible and aperiodic transition matrix considered in \cite{even2023stochastic}. The mixing time $\tau_{\text{mix}}$ is the number of steps of the Markov chain required for the distribution of the current state to be close to the stationary probability $\pi$. We now provide an assumption on the stochastic gradient.
\captionsetup{type=assumption}
\renewcommand{\theassumption}{4}
\begin{assumption}
    \label{as:var}
    For all $x\in\R^d$ it holds that $\E_{\pi}[\nabla F(x,Z)] = \nabla f(x)$. Moreover, for all $Z \in \cZ$ and $x\in \R^d$ the following inequality holds
    \begin{equation*}
        \| \nabla F(x, Z) - \nabla f(x)\|_*^2 \leq \sigma^2.
    \end{equation*}
\end{assumption}
In the Markovian noise problems, we are forced to bound the noise uniformly rather than in expectation, as is done in the i.i.d. case \cite{agarwal2011stochastic, bach2016highly, akhavan2020exploiting, dvurechensky2021accelerated}. This complication is found in many works \cite{duchi2012ergodic, sun2018markov, doan2020convergence, doan2022finite, dorfman2022adapting, even2023stochastic, beznosikov2024first}, and the authors have not yet worked out how to avoid this assumption.

In order to obtain well-founded theoretical estimates, we will consider accelerated methods to analyze the Markovian noise setting. We will present two different techniques for gradient estimation. One technique utilizes a batching construction similar to that used in \cite{dorfman2022adapting, beznosikov2024first}. The other technique does not involve any batching. Firstly, let us introduce a method that uses only one sample of $Z$ at each iteration (Algorithm \ref{alg:MDGwb}).

\begin{algorithm}[h!]
   \caption{\texttt{Markovian Accelerated Mirror Descent (MAMD)} without batching}
   \label{alg:MDGwb}
\begin{algorithmic}[1]
    \State {\bf Parameters:} stepsizes $\{\gamma_t\}$, momentums $\{\beta_t\}$ and number of iterations $T$.
    \State {\bf Initialization:} choose  $x^0 = x_f^0 \in \mathcal{X}$.
    \For{$t = 0, 1, 2, \dots, T$}
        \State $x_g^t = \beta_t^{-1}x^t+(1-\beta_t^{-1})x_f^t$ 
        \State $x^{t+1} = 
        P_{x^t}(\gamma_t \nabla F(x_g^t, Z_t))$
        \State $x_f^{t+1} = \beta^{-1}_t x^{t+1} +(1 - \beta^{-1}_t)x_f^t$
    \EndFor
\end{algorithmic}
\end{algorithm}

This algorithm is similar to vanilla Accelerated Mirror Descent \cite{lan2012optimal}, but with the Markovian noise. The convergence rate of Algorithm \ref{alg:MDGwb} is explored in the theorem below.
\begin{theorem}[Convergence of \texttt{MAMD} without batching (Algorithm \ref{alg:MDGwb})]
\label{theorem:MDGwb}
    Let Assumptions \ref{as:lip}, \ref{as:conv}, \ref{as:noise}, \ref{as:var} be satisfied. Let the problem \eqref{eq:problem1} be solved by Algorithm \ref{alg:MDGwb}. Assume that the stepsizes $\gamma_t$ and momentums $\beta_t$ are chosen such that $0 \leq (\beta_{t+1}-1)\gamma_{t+1} \leq \beta_t\gamma_t$, $\beta_t \geq 2 \gamma_t L$ for all $t \geq \tau_{\text{mix}}$ and $\beta_{\tau_{\text{mix}}} = 1$. Then, for all $T \geq \tau_{\text{mix}}$ it holds that 
    \begin{equation*}
        (\beta_{T} - 1) \gamma_{T} \expect{f(x^{T}_f) - f^*}
        = \widetilde{\cO} \left(
        D^2
        +
        \tau_{\text{mix}}^3 \sigma^2 \sum\limits_{t=\tau}^T \gamma_t^2. \right) .
    \end{equation*}
\end{theorem}

We can specify the choice of $\gamma_t$ and $\beta_t$ to achieve the accelerated convergence of Algorithm \ref{alg:MDGwb}. If we take $\beta_t = \max\{(t-\tau_{\text{mix}}) / 2 + 1; 1\}$ and $\gamma_t \leq \beta_t/(2L)$, then all conditions on the stepsizes and momentums from Theorem \ref{theorem:MDGwb} are satisfied.
\begin{corollary}[Parameters tuning for Theorem \ref{theorem:MDGwb}]
\label{corollary:MDGwb}
    Under the conditions of Theorem \ref{theorem:MDGwb}, choosing $\beta_t$ and $\gamma_t$ as 
    \begin{equation*}
        \beta_t := \max\left\{ \frac{t - \tau_{\text{mix}}}{2} + 1;1 \right\} 
        ,
        \gamma_t := \max\left\{ \frac{t - \tau_{\text{mix}}}{2} + 1;1 \right\} \cdot 
        \min\left\{ \frac{1}{2L} ~;~ \frac{D}{(T - \tau_{\text{mix}})^{3/2} \sigma \tau_{\text{mix}}^{3/2} }\right\},
    \end{equation*}

    in order to achieve the $\varepsilon$-approximate solution (in terms of $\expect{f(x^{T}_f) - f^*} \leq \varepsilon$) it takes

    \begin{equation*}
        T = \widetilde{\cO} \left( \max\left\{ \sqrt{\frac{L D^2}{\varepsilon}} ~;~ \frac{\tau_{\text{mix}}^3 D^2 \sigma^2}{\varepsilon^2} \right\} \right) 
        \text{ iterations (oracle calls) of Algorithm \ref{alg:MDGwb}} .
    \end{equation*}
\end{corollary}

The full proofs of Theorem \ref{theorem:MDGwb} and Corollary \ref{corollary:MDGwb} are provided in Appendix \ref{appendix:MDGwb}. The results of Theorem \ref{theorem:MDGwb} and Corollary \ref{corollary:MDGwb} match with the results from \cite{lan2012optimal} in the i.i.d. ($\tau_{\text{mix}} = 1$) case. In this work, the authors obtained results of the form $T = \cO \left( \max\left\{ \sqrt{(L  D^2) / \varepsilon} ~;~ D^2 \sigma^2 / \varepsilon^2 \right\} \right)$, but in the Markovian noise setup we inevitably face the terms of the form $poly(\tau_{\text{mix}})$. These terms cannot be removed, because they appear in the lower bounds on the convergence rate of the methods that utilize Markovian properties \cite{bresler2020squares}. Despite this, Algorithm \ref{alg:MDGwb} has a reasonably good polynomial dependence on mixing time $\tau_{\text{mix}}$. However, there are several works \cite{doan2020finitetime, doan2022finite}, whose bounds have terms that are even exponential in the mixing time.

\textbf{Sketch proof of Theorem \ref{theorem:MDG}.} Since we did not use the batching technique, in the proof of Theorem \ref{theorem:MDGwb} we were forced to use the technique of \textit{stepping back} by $\tau_{\text{mix}}$ steps, i.e., when estimating the term of the form  $\expect{\dotprod{\nabla f(x^t_g) - \nabla F(x^t_g, Z_t)}{x^{t} - x^*}}$, which is zero in the i.i.d. case, we need to reel it in the form of

\begin{equation*}
\begin{split}
    \expect{\dotprod{\nabla f(x^t_g) - \nabla F(x^t_g, Z_t)}{x^{t} - x^*}}
    &=
    \expect{\dotprod{\nabla f(x^{t}_g) - \nabla F(x^{t}_g, Z_t)}{x^{t-\tau_{\text{mix}}} - x^*}}
    \\&+
    \sum\limits_{s = 0}^{\tau_{\text{mix}} - 1}\expect{\dotprod{\nabla f(x^{t}_g) - \nabla F(x^{t}_g, Z_t)}{x^{t-s} - x^{t-s-1}}} .
\end{split}
\end{equation*}

The first term is bounded using the ergodicity of the Markov chain $\{Z_t\}_{t=0}^\infty$ (Assumption \ref{as:noise}), and the second one is bounded by the Fenchel-Young inequality and Assumption \ref{as:var}. For a more detailed proof, see Appendix \ref{appendix:MDGwb}. 

If we use the batching technique, as in \cite{dorfman2022adapting, beznosikov2024first}, we can avoid the complicated schemes and further reduce the polynomial dependence on $\tau_{\text{mix}}$ to a linear one. Let us now introduce a modified version of Algorithm \ref{alg:MDGwb}, that utilizes the batching technique (Algorithm \ref{alg:MDG}).

\begin{algorithm}[h!]
   \caption{\texttt{Markovian Accelerated Mirror Descent (MAMD)} with batching}
   \label{alg:MDG}
\begin{algorithmic}[1]
    \State {\bf Parameters:} stepsizes $\{\gamma_t\}$, momentums $\{\beta_t\}$, number of iterations $T$, batchsize $B$ and batchsize limit $M$.
    \State {\bf Initialization:} choose  $x^0 = x_f^0 \in \mathcal{X}$ .
    \For{$t = 0, 1, 2, \dots, T$}
        \State $x_g^t = \beta_t^{-1}x^t+(1-\beta_t^{-1})x_f^t$ 
        \State Sample $\textstyle{J_t \sim \text{Geom}\left(1/2\right)}$
        \State
        \text{\small{ 
        $g^{t} = g^{t}_0 +
            \begin{cases}
            \textstyle{2^{J_t}\hspace{-0.1cm}\left( g^{t}_{J_t}  - g^{t}_{J_t - 1} \right)}, &\hspace{-0.25cm} \text{if } 2^{J_t} \leq M \\
            0, & \hspace{-0.25cm}\text{otherwise}
            \end{cases}
        $~~with
        $
        \textstyle{g^t_j = 2^{-j}B^{-1} \sum\limits_{i=1}^{2^jB} \nabla F
        (x^{t}_g, Z_{N_{t} + i})}
        $
        }} \label{line:g_k}
        \State $x^{t+1} = P_{x^t}(\gamma_t g^t) $
        \State $N_{t+1} = N_t + 2^{J_{t}}B$
        \State $x_f^{t+1} = \beta^{-1}_t x^{t+1} +(1 - \beta^{-1}_t)x_f^t$
    \EndFor
\end{algorithmic}
\end{algorithm} 

Before analyzing the convergence of Algorithm \ref{alg:MDG}, let us provide two important lemmas. The first one is generalizing the results for batching on Markovian noise for arbitrary norms, and the second one, with the help of the first, estimates the closeness of the gradient estimator $g^k$ used in line \ref{line:g_k} of Algorithm \ref{alg:MDG} and the true gradient value $\nabla f(x^t_g)$.

\begin{lemma}
    \label{lem:xuivjopeentertainment}
    Let $\{\xi^t\}_{t = 0}^{\infty}$ be an arbitrary Markov chain, such that it satisfies Assumption \ref{as:noise} and 
    $$
        \E_{\pi}[\xi^t] = 0 ~\text{ and }~ \|\xi^t\|^2_* \leq \sigma^2.
    $$
    Then for any $N \geq 1$ it holds that 
    \begin{equation*}
        \textstyle{\E\left[\left\|N^{-1}\sum_{t = 1}^N\xi^t\right\|^2_*\right] \lesssim N^{-1}R^2\sigma^2 \tau_{\text{mix}}},
    \end{equation*}
    where $R^2 := \max\limits_{x \in \R^d : \|x\| = 1} \hat{V}(0, x)$ and $\hat{V}$ is an arbitrary Bregman divergence, not necessary $V$ on $\cX$. 
\end{lemma}
The estimation obtained in Lemma \ref{lem:xuivjopeentertainment} generalizes the result for the batching technique to Markovian noise in arbitrary norm. We again obtain terms of the form $\textit{poly}(\tau_{\text{mix}})$, which are related to the Markovianness of the noise variables $\{\xi^t\}_{t=0}^\infty$ and the term $R^2$, which is related to the norm, on which the noise variance is bounded, and it can be removed only for certain types of norms, which will be announced below. This result is novel in the literature, since it was not proven even for the i.i.d. case, and in the previous works authors consider only Euclidean norm \cite{beznosikov2024first, paulin2015concentration} with an arbitrary ergodic Markov chain. 

The term with $R^2$ could be effectively bounded if we consider the norm $\| \cdot \| = \| \cdot \|_p$, where $1 \leq p \leq 2$. In this case, there exists a Bregman divergence $\hat{V}$ such that $R^2 \lesssim \log d$, where $d$ is the dimensionality of $x$. If $p > 2$, then we can only conclude that $R^2 \lesssim d^{1/2 - 1/p}$ \cite{ben2001lectures}. 
However, the usage of $\| \cdot \|_{p > 2}$ norm makes no practical sense, because in this case the value of $L$ in Assumption \ref{as:lip} would be bigger then in the classical Euclidean norm, when $\| \cdot \| = \| \cdot \|_2$.
If we consider $\| \cdot \| = \| \cdot \|_{p}$, $1 \leq p \leq 2$, then from Lemma \ref{lem:xuivjopeentertainment} one can obtain the result of the form
\begin{equation*}
    \textstyle{\E\left[\left\|N^{-1}\sum_{t = 1}^N\xi^t\right\|^2_q\right] 
    \lesssim
    \log(d) N^{-1} \tau_{\text{mix}} \sigma^2,}
\end{equation*}
where $1/p + 1/q = 1$. This result matches with the result for Euclidean norm in the \cite{beznosikov2024first} (in this case $\hat{V}(x, y) = 1/2\|x - y\|_2^2$ and $R^2 = 1/2$, therefore $\log(d)$ disappears), but for $1 \leq p \lesssim \log(d)^{-1}$, the term $\log(d)$ can not be removed \cite{ben2001lectures}. We regard this as a slight charge for generalizing the result to an arbitrary norm and further we omit this term in our theoretical estimates.
\begin{lemma}
\label{lem:expect_bound_grad}
Consider Assumptions \ref{as:noise} and \ref{as:var}. Then for the gradient estimates $g^t$ from line \ref{line:g_k} of Algorithm \ref{alg:MDG}, it holds that $\E_t[g^t] = \E_t[g^{t}_{\lfloor \log M \rfloor}]$. Moreover, 
\begin{equation*}
\begin{split}
    &\E_t[\| \nabla f(x^t_g) - g^t\|_*^2] \lesssim B^{-1}\tau_{\text{mix}}\log (M) \sigma^2,\\
    &\| \nabla f(x^t_g) - \E_{t}[g^t]\|_*^2 \lesssim B^{-1}\tau_{\text{mix}} M^{-1} \sigma^2.
\end{split}
\end{equation*}
\end{lemma}
The full proofs of Lemmas \ref{lem:xuivjopeentertainment} and \ref{lem:expect_bound_grad} are provided in Appendix \ref{appendix:two_lemmas}.
Note that Lemma \ref{lem:expect_bound_grad} is a natural continuation of Lemma \ref{lem:xuivjopeentertainment}. Moreover, it yields intuition about the trade-off between parameters $M$ and $B$, that control number of calls to the first-order oracle $\nabla F(x, Z)$ at each iteration. The expectation of oracle calls needed to compute $g^t$ is equal to $\mathcal{O}(B \log(M))$, thus $M$ affects the batchesize as a logarithm, but it decreases the bias term as $M^{-1}$. Hence, increasing this parameter is not significantly expensive in terms of computational complexity, but is very helpful in terms of accuracy. Parameter $B$ polynomially increases the number of calculations to compute $g^t$, therefore we take $B = 1$ in our algorithms. We are now ready to explore the convergence rate for Algorithm \ref{alg:MDG}.
\begin{theorem}[Convergence of \texttt{MAMD} with batching (Algorithm \ref{alg:MDG})]
    \label{theorem:MDG}
    Let Assumptions \ref{as:lip}, \ref{as:conv}, \ref{as:noise}, \ref{as:var} be satisfied with $\| \cdot \| = \| \cdot \|_p$, $1 \leq p \leq 2$. Let the problem \eqref{eq:problem1} be solved by Algorithm \ref{alg:MDG}. Assume that the stepsizes $\gamma_t$ and momentums $\beta_t$ are chosen such that $0 \leq (\beta_{t+1}-1)\gamma_{t+1} \leq \beta_t\gamma_t$, $\beta_t \geq 2 \gamma_t L$ for all $t \geq 0$ and $\beta_0 = 1$. Then, for all $T \geq 0$ it holds that 
    \begin{equation*}
            (\beta_T - 1) \gamma_T \expect{f(x^{T}_f) - f^*}
            = \widetilde{\cO} \left(
            D^2 
            +
            \tau_{\text{mix}} B^{-1} \left( T M^{-1} + \log M \right) \sigma^2 \sum\limits_{t = 0}^{T-1}\gamma_t^2 \right).
    \end{equation*}
\end{theorem}

Similarly, as outlined in Corollary \ref{corollary:MDGwb}, we can specify the choice of all parameters of Algorithm \ref{alg:MDG}.

\begin{corollary}[Parameters tuning for Theorem \ref{theorem:MDG}]
\label{corollary:MDG}
    Under the conditions of Theorem \ref{theorem:MDG}, choosing $\beta_t$, $\gamma_t$, $M$ and $B$ as 
    \begin{equation*}
        \beta_t := \frac{t}{2} +  1
        ,~
        \gamma_t := \left(\frac{t}{2} +  1 \right) \cdot 
        \min\left\{ \frac{1}{2L} ~;~ \frac{D}{T^{3/2} \sigma \tau_{\text{mix}}^{1/2} }\right\}
        ,~
        M :=  T 
        ~\text{ and }~
        B := 1,
    \end{equation*}
    in order to achieve the $\varepsilon$-approximate solution (in terms of $\expect{f(x^{T}_f) - f^*} \leq \varepsilon$) it takes
    \begin{equation*}
        T = \widetilde{\cO} \left( \max\left\{ \sqrt{\frac{L D^2}{\varepsilon}} ~;~ \frac{\tau_{\text{mix}} D^2 \sigma^2}{\varepsilon^2} \right\} \right) 
        \text{ iterations (oracle calls) of Algorithm \ref{alg:MDG}} .
    \end{equation*}
\end{corollary}
The full proofs of Theorem \ref{theorem:MDG} and Corollary \ref{corollary:MDG} are provided in Appendix \ref{appendix:MDG}. 
\vspace{0.15cm}

\textbf{Discussion. }In Corollary \ref{corollary:MDG}, we get the result directly in terms of oracle complexity, since using $g^t$ as a gradient estimator requires to make $\cO(B \log(M))$ first-order oracle calls at each iteration. And since in Corollary \ref{corollary:MDG} we set $M = T$ and $B = 1$, we can conclude that oracle complexity is equal to iteration complexity in terms of $\widetilde{\cO}(\cdot)$.
The results of Theorem \ref{theorem:MDG} and Corollary \ref{corollary:MDG} again match with the results from \cite{lan2012optimal} in the i.i.d. ($\tau_{\text{mix}} = 1$) case. The usage of Markovian batching with batchsize $B = \widetilde{\cO}(1)$ (Algorithm \ref{alg:MDG}) helps us to achieve better performance compared to Algorithm \ref{alg:MDGwb} (see Corollary \ref{corollary:MDGwb}). The degree of the dependence on $\tau_{\text{mix}}$ in the stochastic term decreases from the cubic to a linear one. If we consider $B \sim \varepsilon, \tau_{\text{mix}}, \sigma^2, L$, then, according to Theorem \ref{theorem:MDG}, we can achieve the same convergence rate as in the deterministic case: $T = \widetilde{\cO}(\sqrt{(L D^2) / \varepsilon})$. However, the oracle complexity would be the same as in Corollary \ref{corollary:MDG}.

Let us now compare the convergence rate for Theorem \ref{theorem:MDG} with the previous works about Markovian noise in the minimization problem of the form \eqref{eq:problem1}. Since in Corollary \ref{corollary:MDG} we obtained the result in terms of oracle complexity, we will compare with the existing methods on this criterion, in case of other algorithms use batch sizes that are not equal to $\widetilde{\cO}(1)$.

\textbf{Comparison.} In \cite{doan2022finite}, the authors get the result of the form $T = \widetilde{\cO}\left( e^{\tau_{\text{mix}} (L^2D^2 / \varepsilon^2)}\right)$. This estimate is much worse than the one observed in Corollary \ref{corollary:MDG}, since there is an exponential dependence on $\tau_{\text{mix}}$ and $\varepsilon$.
The authors of \cite{doan2020convergence} obtain $T = \widetilde{\cO}\left( \max\left\{\sqrt{L} D^3 / \varepsilon ~;~ \tau_{\text{mix}}^2 D^2 G^2 / \varepsilon^2 \right\}\right)$, where $G^2 := \max_{x, Z}\{\|\nabla f(x, Z)\|^2\}$. 
A deterministic term is again worse than the one observed in Corollary \ref{corollary:MDG}, and the stochastic term is multiplied by $\tau_{\text{mix}}^2$, instead of $\tau_{\text{mix}}$, as it is for our Algorithm \ref{alg:MDG}. 
Moreover, $G^2$ bounds the stochastic gradient uniformly, and can be significantly larger than $L^2 D^2$ and $\sigma^2$ from Assumption \ref{as:var}.
The paper \cite{zhao2023markov} provided the guarantee of the form $T = \widetilde{\cO}\left( \max\left\{ \tau_{\text{mix}}^2 L^2 D^2 / \varepsilon ~;~ \tau_{\text{mix}} D^2 G^2 / \varepsilon^2 \right\}\right)$. The deterministic term is considerably worse than the one presented in Corollary \ref{corollary:MDG}. In the stochastic term, there is again a uniform bound of the gradient norm by $G^2$.
In Section 5.1 of the work \cite{even2023stochastic}, authors observe the result $T = \widetilde{\cO}\left(\tau_{\text{mix}} \max\left\{L D^2 / \varepsilon ~;~ D^2 \sigma^2 / \varepsilon^2 \right\}\right)$. A dependence on $\varepsilon^{-1}$ in the deterministic term is observed, since accelerated methods were not considered in the \cite{even2023stochastic}. It can be seen that both terms are multiplied by $\tau_{\text{mix}}$, whereas in Corollary \ref{corollary:MDG} only the stochastic term is. In Section 5.2 of the same work \cite{even2023stochastic}, the authors were able to circumvent Assumption \ref{as:var}. Instead of these, they assume uniform boundedness of the noise at the solution of the problem \eqref{eq:problem1} only: $\|\nabla F(x^*, Z) - \nabla f(x^*)\|^2 \leq \sigma^2_*$ for all $Z \in \cZ$. However, to achieve this, they utilized the fact that all stochastic realizations $F(x, Z)$ are $L$-smooth and convex, which is a more strict compared to Assumptions \ref{as:lip} and \ref{as:conv}. The authors of \cite{beznosikov2024first} presented the estimate of the form $T = \widetilde{\cO}\left(\tau_{\text{mix}} \max\left\{\sqrt{L D^2 / \varepsilon} ~;~ D^2 \sigma^2 / \varepsilon^2 \right\}\right)$. In this work, deterministic and stochastic parts are the same as in Corollary \ref{corollary:MDG}, but both terms are multiplied by $\tau_{\text{mix}}$ again, because authors used batch of size $\widetilde{\cO}(\tau_{\text{mix}})$, instead of $\widetilde{\cO}(1)$. Equally important, all of these works considered only the Euclidean setup, as mentioned above in Section \ref{section:introduction}.
\vspace{0.1cm}

\textbf{Lower bound.} Combining two lower bounds from \cite{nesterov2013introductory} (for the deterministic term) and from \cite{duchi2012ergodic} (for the stochastic term), we provide a result of the form.

\begin{proposition}[Lower bound for \eqref{eq:problem1}]
\label{proposition:1}
    There exists an instance of the optimization problem \eqref{eq:problem1} satisfying Assumptions \ref{as:lip}, \ref{as:conv}, \ref{as:noise}, \ref{as:var} with arbitrary $L > 0, \sigma^2 \geq 0, \tau_{\text{mix}} \in \N$, such that for any stochastic first-order gradient method it takes at least
    \begin{equation*}
        T = \Omega \left( \max\left\{ \sqrt{\frac{L D^2}{\varepsilon}} ~;~ \frac{\tau_{\text{mix}} D^2 \sigma^2}{\varepsilon^2}\right\} \right)
    \end{equation*}
    oracle calls in order to achieve $\expect{f(x^T) - f^*} \leq \varepsilon$.
\end{proposition}
The full proof of Proposition \ref{proposition:1} is provided in Appendix \ref{appendix:lover_bounds}. As follows from Proposition \ref{proposition:1} and Corollary \ref{corollary:MDG}, Algorithm \ref{alg:MDG} is optimal in the class of first-order methods with Markovian noise.
\subsection{Markovian Mirror-Prox}
\label{section:MP}
In this section, we are interested in the optimization problem of the following form
\begin{equation}
    \label{eq:problem2}
    \text{Find } x\in\mathcal{X}\text{ such that } \langle F(x^*), x - x^*\rangle \geq 0 \text{ for all } x\in\mathcal{X}.
\end{equation}
In this equation, $F:\mathcal{X}\rightarrow\mathcal{Y}$ is an operator and $\mathcal{X}$ is a compact convex set in the Euclidean space $\mathcal{Y}$. In this setup, we also assume that the access to the objective operator $F(x)$ is available only through the noisy oracle $F(x, Z)$, i.e., $F(x) = \E_\pi[F(x, Z)]$. For convex minimization, $F$ is the gradient of the objective function, while for the convex-concave saddle point problem $F$ is composed of gradient and negative gradient of the objective with respect to the primal and dual variables. We now provide several assumptions required for the analysis. Assumptions \ref{as:lipvar} and \ref{as:monotone} are akin to Assumptions \ref{as:lip} and \ref{as:conv} for the $\nabla f(x)$ in the minimization problem, but in the variational inequalities we generalize them for an arbitrary operator $F(x)$.
\captionsetup{type=assumption}
\renewcommand{\theassumption}{5(a)}
\begin{assumption}
    \label{as:lipvar_Z}
    The operator $F(x, Z)$ is $L(Z)$-Lipschitz continuous on $\mathcal{X}$, i.e., for any $x, y \in \mathcal{X}$, for any $Z \in \cZ$ the following inequality holds 
    \begin{equation*}
        \|F(x, Z) - F(y, Z)\|_*\leq L(Z) \|x - y\|,
    \end{equation*}
    we also denote $\Tilde{L} := \sup_{Z\in\cZ}L(Z)$.
\end{assumption}
\captionsetup{type=assumption}
\renewcommand{\theassumption}{5(b)}
\begin{assumption}
    \label{as:lipvar}
    The operator $F$ is $L$-Lipschitz continuous on $\mathcal{X}$, i.e., for any $x, y \in \mathcal{X}$ the following inequality holds 
    \begin{equation*}
        \|F(x) - F(y)\|_*\leq L \|x - y\|.
    \end{equation*}
\end{assumption}
\captionsetup{type=assumption}
\renewcommand{\theassumption}{6(a)}
\begin{assumption}
    \label{as:monotone_Z}
    The operator $F(\cdot, Z)$ is monotone on $\mathcal{X}$, i.e., for all $x, y \in \mathcal{X}$ and for any $Z \in \cZ$ the following inequality holds 
    \begin{equation*}
        \langle F(x, Z) - F(y, Z), x - y\rangle \geq 0.
    \end{equation*}
\end{assumption}
\captionsetup{type=assumption}
\renewcommand{\theassumption}{6(b)}
\begin{assumption}
    \label{as:monotone}
    The operator $F$ is monotone on $\mathcal{X}$, i.e., for all $x, y \in \mathcal{X}$ the following inequality holds 
    \begin{equation*}
        \langle F(x) - F(y), x - y\rangle \geq 0.
    \end{equation*}
\end{assumption}
In the absence of batching (Algorithm \ref{alg:MPGwb}), we posit that for $Z \in \cZ$ operators $F(x, Z)$ are $L$-Lipschitz continuous and monotone. In the case of the batched version (Algorithm \ref{alg:MPG}), we assume it only for the objective operator $F(x)$. The aforementioned assumptions in Algorithm \ref{alg:MPGwb} are introduced in order to bound the noise not uniformly (as in Assumption \ref{as:var}), but in expectation over the stationary distribution $\pi$. At the same time, for most of the noise designs presented in the literature \cite{mania2017perturbed, holmstrom1992using}, Lipschitzness and monotonicity across realizations of a noise variable are fulfilled by the construction, therefore these assumptions remain relevant to the analysis. In the case of Lipschitzness and monotonicity of the noise-free (objective) operator only, we are forced to bound the noise uniformly, similar to what we did it in Section \ref{subsection:MAMD}.
\captionsetup{type=assumption}
\renewcommand{\theassumption}{7(a)}
\begin{assumption}
    \label{as:varvar_pi}
    For all $x\in\R^d$ it holds that $\E_{\pi}[F(x,Z)] = F(x)$. Morover, for all $x\in \R^d$ the following inequality holds 
    \begin{equation*}
        \E_{\pi}[ \| F(x, Z) - F(x)\|^2_*] \leq \sigma^2.
    \end{equation*}
\end{assumption}
\captionsetup{type=assumption}
\renewcommand{\theassumption}{7(b)}
\begin{assumption}
    \label{as:varvar_uniform}
    For all $x\in\R^d$ it holds that $\E_{\pi}[F(x,Z)] = F(x)$. Moreover, for all $Z \in \cZ$ and $x\in \R^d$ the following inequality holds 
    \begin{equation*}
        \| F(x, Z) - F(x)\|_*^2 \leq \sigma^2.
    \end{equation*}
\end{assumption} 
We are now ready to introduce a novel algorithm that uses only one sample of the Markov chain at each iteration (Algorithm \ref{alg:MPGwb}).
\begin{algorithm}[h!]
   \caption{\texttt{Markovian Mirror-Prox (MMP)} without batching}
   \label{alg:MPGwb}
\begin{algorithmic}[1]
    \State {\bf Parameters:} stepsize $\gamma>0$, number of iterations $T$
    \State {\bf Initialization:} choose  $x^0 \in \cX$
    \For{$t = 0, 1, 2, \dots, T$}
        \State $x^{t+1/2} = P_{x^t} \left(\gamma F(x^t, Z_{t}) \right)$
        \State    $x^{t+1} = P_{x^{t}} \left(\gamma F(x^{t + 1/2}, Z_t) \right)$
    \EndFor
\end{algorithmic}
\end{algorithm}

\begin{theorem}[Convergence of \texttt{MMP} without batching (Algorithm \ref{alg:MPGwb})]
    \label{theorem:MPGwb}
    Let Assumptions \ref{as:noise}, \ref{as:lipvar_Z}, \ref{as:monotone_Z}, \ref{as:varvar_pi} be satisfied. Let the problem \eqref{eq:problem2} be solved by Algorithm \ref{alg:MPGwb}. Assume that the stepsize $\gamma$ is chosen such that $\gamma \leq 1/(2L)$. Then, for all $T \geq \tau_{\text{mix}}$ it holds that  
    \begin{equation*}
        \underset{u \in \cX}{\max} \left\{ \E[\langle F(u), \widehat{x}^{T} - u\rangle] \right\} \leq \frac{2D^2}{\gamma (T - \tau_{\text{mix}})} + 12\gamma \tau_{\text{mix}}^2 \sigma^2,
    \end{equation*}
    where $\widehat{x}^{T} = \frac{1}{T - \tau_{\text{mix}}}\sum\nolimits_{t = \tau_{\text{mix}}}^{T-1}x^{t+\frac{1}{2}}$.
\end{theorem}
\begin{corollary}[Parameters tuning for Theorem \ref{theorem:MPGwb}]
    \label{cor:th3}
    Under the conditions of Theorem \ref{theorem:MPGwb}, choosing $\gamma$ as $$\gamma \leq \min\left\{ \frac{1}{2L} ~;~ \frac{D}{(T - \tau_{\text{mix}})^{1/2} \sigma \tau_{\text{mix}}} \right\},$$ in order to achieve the $\varepsilon$-approximate solution (in terms of $\underset{u \in \cX}{\max} \left\{ \E[\langle F(u), \widehat{x}^{T} - u\rangle] \right\} \leq \varepsilon$) it takes \begin{equation*}
        T = \Tilde{\mathcal{O}}\Bigg(\max\left\{\frac{L D^2}{\varepsilon} ~;~ \frac{\tau_{\text{mix}}^2 D^2 \sigma^2}{\varepsilon^2} \right\}\Bigg) \text{ iterations (oracle calls) of Algorithm \ref{alg:MPGwb}}.
    \end{equation*}
\end{corollary}
The full proofs of Theorem \ref{theorem:MPG} and Corollary \ref{cor:th3} are provided in Appendix \ref{appendix:th3}.
To proceed to the next algorithm, we first introduce a more powerful convergence criterion.  
Now, drawing on \cite{Nesterov2003DualEA, Nesterov2005PrimaldualSM} and \cite{juditsky2011solving}, if $F$ is monotone, the quality of a candidate solution $x \in \cX$ can be assessed via the \textit{error (or merit)} function
\begin{equation}
    \label{def:err_vi}
    \text{Err}_{\text{VI}}(x) := \underset{u\in \cX}{\max}\langle F(u), x - u\rangle.
\end{equation}
The rationale behind this definition is that, if $x^*$
solves \eqref{eq:problem2}, monotonicity gives $\langle F(u), x^* - u \rangle \leq \langle F(x^*), x^* - u\rangle \leq 0$, therefore the quantity being maximized
in \eqref{def:err_vi} is small if $x$ is an approximate solution of \eqref{eq:problem2}. It has also been shown in \cite{Nesterov2003DualEA, antonakopoulos2019adaptive} that $\text{Err}_{\text{VI}}(x) = 0$ if and only if $x$ is a solution of the problem \eqref{eq:problem2}. Since the specificity of working with random variables implies stochastic convergence, we further use $\E[\text{Err}_{\text{VI}}(x)]$. Note also that the convergence criterion used in Theorem \ref{theorem:MPGwb} satisfies the inequality $\max_{u \in \cX} \left\{ \E[\langle F(u), x - u\rangle] \right\} \leq \E[\text{Err}_{\text{VI}}(x)] \text{ for all } x\in\cX$. The use of $\max\{\expect{\cdot}\}$ allows us to bound the gradient variance in Theorem \ref{theorem:MPGwb} only in the mathematical expectation -- similar to the i.i.d. case, which, to the best of our knowledge, was never done before in the Markovian noise setup. On the other hand, the uniformly bounded noise (Assumption \ref{as:lipvar_Z}) enables to utilize a more stringent convergence criterion and less stringent Assumptions \ref{as:lipvar} and \ref{as:monotone}. Let us now introduce a modified version of Algorithm \ref{alg:MPGwb}, that utilizes the batching technique for Markovian noise (Algorithm \ref{alg:MPG}). 

\begin{algorithm}[h!]
   \caption{\texttt{Markovian Mirror-Prox (MMP)} with batching}
   \label{alg:MPG}
\begin{algorithmic}[1]
    \State {\bf Parameters:} stepsize $\gamma>0$, number of iterations $T$
    \State {\bf Initialization:} choose  $x^0 \in \cX$
    \For{$t = 0, 1, 2, \dots, T$}
        \State $g^{t+1/2} = B^{-1}\sum\limits_{i=1}^{B}F(x^t, Z_{N_t+i})$
        \State $x^{t+1/2} = P_{x^t} \left( \gamma g^{t+1/2} \right)$
        \State $N_{t+1/2} = N_t + B$
        \State Sample $\textstyle{J_t \sim \text{Geom}\left(1/2\right)}$
        \State
        \text{\small{ 
        $g^{t} = g^{t}_0 +
            \begin{cases}
            \textstyle{2^{J_t}\hspace{-0.1cm}\left( g^{t}_{J_t}  - g^{t}_{J_t - 1} \right)}, &\hspace{-0.25cm} \text{if } 2^{J_t} \leq M \\
            0, & \hspace{-0.25cm}\text{otherwise}
            \end{cases}
        $~~with
        $
        \textstyle{g^t_j = 2^{-j}B^{-1} \sum\limits_{i=1}^{2^jB} F
        (x^{t+1/2}, Z_{N_{t+1/2} + i})}
        $
        }}  
        \State $x^{t+1} = P_{x^t} \left( \gamma g^t \right)$
        \State $\textstyle{N_{t+1} = N_{t+1/2} + 2^{J_{t}}B}$ 
    \EndFor
\end{algorithmic}
\end{algorithm}

Now we explore the convergence rate for Algorithm \ref{alg:MPG}.
\begin{theorem}[Convergence of MMP]
    \label{theorem:MPG}
    Let Assumptions \ref{as:noise}, \ref{as:lipvar}, \ref{as:monotone}, \ref{as:varvar_uniform} be satisfied with $\| \cdot \| = \| \cdot \|_p$, $1 \leq p \leq 2$. Let the problem \eqref{eq:problem2} be solved by Algorithm \ref{alg:MPG}. Assume that the stepsize $\gamma$ is chosen such that $\gamma \leq 1/(2L)$. Then, for all $T \geq 0$ it holds that 
    \begin{equation*}
            \expect{\text{{\normalfont Err}}_{\text{{\normalfont VI}}}(\widehat{x}^{T})} 
            = 
            \widetilde{\cO} \left( 
                \frac{D^2}{\gamma T}
                +
                \gamma \tau_{\text{mix}} B^{-1} \left( TM^{-1} + \log(M) \right) \sigma^2
            \right).
        \end{equation*}
\end{theorem}
\begin{corollary}[Parameters tuning for Theorem \ref{theorem:MPG}]
    \label{cor:th4}
    Under the conditions of Theorem \ref{theorem:MPG}, choosing  $\gamma$, $M$ and $B$ as 
    \begin{equation*}
        \gamma := \min\left\{ \frac{1}{2L} ~;~ \frac{D}{T^{1/2} \sigma \tau_{\text{mix}}^{1/2} }\right\}
        ,~
        M :=  T 
        ~\text{ and }~
        B := 1,
    \end{equation*}
    in order to achieve the $\varepsilon$-approximate solution (in terms of $\expect{\text{{\normalfont Err}}_{\text{{\normalfont VI}}}(\widehat{x}^{T})} \leq \varepsilon$) it takes \begin{equation*}
        T = \widetilde{\mathcal{O}}\left( \max\left\{\frac{L D^2}{\varepsilon} ~;~ \frac{\tau_{\text{mix}} D^2 \sigma^2}{\varepsilon^2} \right\} \right) \text{ iterations (oracle calls) of Algorithm \ref{alg:MPG}}.
    \end{equation*}
\end{corollary}
The full proofs of Theorem \ref{theorem:MPG} and Corollary \ref{cor:th4} are provided in Appendix \ref{appendix:th4}. 
\vspace{0.15cm}

\textbf{Discussion. }In Corollary \ref{cor:th4}, we again obtain the result in terms of oracle complexity, since the usage $g^t$ requires $\cO(B \log(M))$ of oracle calls at each iteration.

The results of Theorem \ref{theorem:MPG} and Corollary \ref{cor:th4} match with the results from \cite{juditsky2011solving} in the i.i.d. ($\tau_{\text{mix}} = 1$) case:  $T = \cO \left( \max\left\{ (L  D^2) / \varepsilon ~;~ D^2 \sigma^2 / \varepsilon^2 \right\} \right)$. The usage of the Markovian batching helps us to achieve better performance compared to Algorithm \ref{alg:MPGwb} (see Corollary \ref{cor:th3}). Firstly, we obtain theoretical guarantees for a more strict criterion of the form $\expect{\text{{\normalfont Err}}_{\text{{\normalfont VI}}}(\widehat{x}^{T})}$. Secondly, the degree of dependence on $\tau_{\text{mix}}$ in the stochastic term decreases from the quadratic to a linear one.

Let us now compare the convergence rate for Theorem \ref{theorem:MPG} with the previous works about the Markovian noise in the variational inequality problem \eqref{eq:problem2}. We again will compare with existing methods on oracle complexity criterion.

\textbf{Comparison.} To the best of our knowledge, there exist only two works on the topic of VI with Markovian stochasticity. In the first one \cite{wang2022stability}, the authors consider only saddle point problems and provide the result of the form $T = \widetilde{\cO}\left( (G^2 + \tau_{\text{mix}}^2 G^4) / \varepsilon^2 \right)$, where $G$ is the uniform bound of the stochastic operator. This estimate is much worse than the one we observe in Corollary \ref{cor:th4}, since the deterministic term contains a $\varepsilon^{-2}$ dependence, and the stochastic term not only contains $G^2$ rather than $\sigma^2$, but also is multiplied by $\tau_{\text{mix}}^2$. The second work \cite{beznosikov2024first} provided the guarantee of the form $T = \widetilde{\cO}\left(\tau_{\text{mix}} \max\left\{L D^2 / \varepsilon ~;~ D^2 \sigma^2 / \varepsilon^2 \right\}\right)$. This result is almost the same as in Corollary \ref{cor:th4}, but both terms are multiplied by $\tau_{\text{mix}}$, because the authors again used the batch of size $\widetilde{\cO}(\tau_{\text{mix}})$, instead of $\widetilde{\cO}(1)$. 

\textbf{Lower bound.} Combining two lower bounds from \cite{ouyang2021lower} (for the deterministic term) and from \cite{duchi2012ergodic} (for the stochastic term) we provide the following result.
\begin{proposition}[Lower bound for \eqref{eq:problem2}]
\label{proposition:2}
    There exists an instance of the optimization problem \eqref{eq:problem2} satisfying Assumptions \ref{as:noise}, \ref{as:lipvar}, \ref{as:monotone}, \ref{as:varvar_uniform} with arbitrary $L > 0, \sigma^2 \geq 0, \tau_{\text{mix}} \in \N$, such that for any stochastic first-order gradient method it takes at least
    \begin{equation*}
        T = \Omega \left( \max\left\{ \frac{L D^2}{\varepsilon} ~;~ \frac{\tau_{\text{mix}} D^2 \sigma^2}{\varepsilon^2}\right\} \right)
    \end{equation*}
    oracle calls in order to achieve $\expect{\text{{\normalfont Err}}_{\text{{\normalfont VI}}}(x^{T})} \leq \varepsilon$.
\end{proposition}
The full proof of Proposition \ref{proposition:2} is provided in Appendix \ref{appendix:lover_bounds}. As follows from Proposition \ref{proposition:2} and Corollary \ref{cor:th4}, Algorithm \ref{alg:MPG} is optimal in the class of convex VI problems \eqref{eq:problem2} with the Markovian noise.

\newpage

\bibliography{main}  
\bibliographystyle{plain}

\newpage
\appendix
\onecolumn
\part*{Supplementary Material}

\tableofcontents
\newpage


\section{Auxiliary Lemmas and Facts}
    \subsection{Convexity of the squared norm}
    \label{axil:squared}
        For all $x_1, ... , x_n \in \mathbb{R}^d$, where $n \in \N$ it holds that
        \begin{equation*}
            \norms{x_1 + x_2 + ... + x_n}^2 \leq n \norms{x_1}^2 + ... + n \norms{x_n}^2 . 
        \end{equation*}
    \subsection{Cauchy–Schwarz inequality}
    \label{axil:cauchy_schwarz}
        For all $x, y \in \mathbb{R}^d$
        \begin{equation*}
            \dotprod{x}{y} \leq \norms{x}\norms{y}_* .
        \end{equation*}
    \subsection{Fenchel-Young inequality}
    \label{axil:fenchel_young}
    For all $x, y \in \mathbb{R}^d$ and $\kappa > 0$
    \begin{equation*}
        2 \dotprod{x}{y} \leq \kappa^{-1} \|x\|^2 + \kappa\|y\|^2_* .
    \end{equation*}
    \subsection{Bregman Divergence properties}
    \begin{lemma}[Lemma 3 form \cite{juditsky2011solving}]
    \label{nemirovski:lem3}
        For every $x \in \mathcal{X}$, the mapping $\xi \mapsto P_x(\xi)$ is a single-valued mapping of $\mathcal{Y}$ onto $\mathcal{X}$, and this mapping is Lipschitz continuous, specifically,
        \begin{equation}
            \label{banan}
            \|P_x(\eta) - P_x(\zeta)\| \leq \|\eta - \zeta\|_*~~~\forall\eta, \zeta\in\mathcal{Y}.
        \end{equation}
        Besides that, for all $u \in \mathcal{X}$, 
        \begin{equation}
            \begin{split}
                &(a) ~~~~~~V(P_x(\zeta), u) \leq V(x, u) + \langle \zeta, u - P_x(\zeta)\rangle - V(x, P_x(\zeta))\\
                &(b)~~~~~~~~~~~~~~~~~~~~~~~~~~~\leq
                V(x, u) + \langle\zeta, u-x\rangle + \frac{\|\zeta\|_*^2}{2}
            \end{split}
        \end{equation}
    \end{lemma}
    \begin{lemma}[Lemma 4 from \cite{juditsky2011solving}]
    \label{nemirovski:lem4}
        Let $x \in \mathcal{X}$, let $\eta, \zeta$ be two points from $\mathcal{Y}$, and let
        \begin{equation*}
            w = P_x(\eta),~~~~r_+ = P_x(\zeta)
        \end{equation*}
        Then for all $u \in \mathcal{X}$ one has, 
        \begin{equation}
            \begin{split}
                V(r_+, u) - V(x, u) &\leq \langle \eta, u - w\rangle + \langle \eta, w - r_+\rangle - V(x, r_+) \\
                &\leq \langle \eta, u - w\rangle + \frac{1}{2}\|\zeta - \eta\|^2_*-\frac{1}{2}\|w - x\|^2. 
            \end{split}
        \end{equation}
    \end{lemma}
\addtocontents{toc}{\protect\setcounter{tocdepth}{2}}
\section{Proofs of results for \texttt{MAMD} without batching (Algorithm \ref{alg:MDGwb})}
\label{appendix:MDGwb}
    We start by introducing technical lemma similar to Lemma 5 in \cite{lan2012optimal}.
    \begin{lemma}
    \label{lemma:MDwb}
        Let $x^t, x_f^t$ and $x_g^t$ are computed according to Algorithm \ref{alg:MDGwb}. Then, for any $x \in \cX$ and $\eta_t > 0$ the following inequality holds:
        \begin{equation*}
        \begin{split}
            \beta_t \gamma_t (f(x^{t+1}_f) - f(x))
            &\leq
            (\beta_t-1)\gamma_t (f(x_f^t) - f(x))
            +
            \gamma_t \dotprod{\nabla f(x^t_g) - \nabla F(x^t_g, Z_t)}{x^{t} - x} 
            \\&+ 
            V(x^t, x) - V(x^{t+1}, x)
            +
            \frac{\gamma_t^2}{2 \eta_t}\| \nabla f(x^t_g) - \nabla F(x^t_g, Z_t) \|_*^2
            \\&-
            \frac{1}{2}\left(1 - \eta_t - \frac{L \gamma_t}{\beta_t} \right)\norms{x^{t+1} - x^t}^2 .
        \end{split}
        \end{equation*}
    \end{lemma}
    \begin{proof}
        Define $d^t := x^{t+1} - x^t$. Then we can obtain
        \begin{equation*}
            x^{t+1}_f - x^t_g = \beta_t^{-1} x^{t+1} + (1 - \beta_t^{-1})x^t_f - x^t_g = \beta_t^{-1} (x^{t+1} - x^t) = \beta_t^{-1} d^t .
        \end{equation*}
        Using this and $L$-smooth of the function $f$ we can obtain that
        \begin{equation}
        \label{eq:tmp_th1_1}
        \begin{split}
            \beta_t \gamma_t f(x^{t+1}_f) 
            &\leq
            \beta_t \gamma_t \left[ f(x^t_g) + \dotprod{\nabla f(x^t_g)}{x^{t+1}_f - x^t_g} + \frac{L}{2}\norms{x_f^{t+1} - x^t_g}^2\right]
            \\&=
            \beta_t \gamma_t \left[ f(x^t_g) + \dotprod{\nabla f(x^t_g)}{x^{t+1}_f - x^t_g}\right] + \frac{L \gamma_t}{2 \beta_t}\norms{d^t}^2
            \\&\leq
            \beta_t \gamma_t \left[ f(x^t_g) + \dotprod{\nabla f(x^t_g)}{x^{t+1}_f - x^t_g}\right]
            +
            V(x^t, x^{t+1}) - \frac{\beta_t - L \gamma_t}{2 \beta_t}\norms{d^t}^2 .
        \end{split}
        \end{equation}
        Where the last inequality follows from the $1$-strongly convexity of the $\omega(\cdot)$: $V(x, y) \geq 1/2 \|x-y\|^2$ for all $x, y \in \cX$. Next we estimate term under $[ \cdot ]$ in the \eqref{eq:tmp_th1_1}:
        \begin{equation*}
        \begin{split}
            \beta_t \gamma_t \Big[ f(x^t_g) &+ \dotprod{\nabla f(x^t_g)}{x^{t+1}_f - x^t_g}\Big]
            \\&=
            \beta_t \gamma_t \left[ f(x^t_g) + \dotprod{\nabla f(x^t_g)}{(1 - \beta_t^{-1})x^t_f + \beta_t^{-1}x^{t+1} - x^t_g}\right]
            \\&=
            (\beta_t-1)\gamma_t \left[ f(x^t_g) + \dotprod{\nabla f(x^t_g)}{x^t_f - x^t_g}\right]
            +
            \gamma_t \left[ f(x^t_g) + \dotprod{\nabla f(x^t_g)}{x^{t+1} - x^t_g}\right]
            \\&\overset{(a)}{\leq}
            (\beta_t-1)\gamma_t f(x_f^t)
            +
            \gamma_t \left[ f(x^t_g) + \dotprod{\nabla f(x^t_g)}{x^{t+1} - x^t_g}\right]
            \\&=
            (\beta_t-1)\gamma_t f(x_f^t)
            +
            \gamma_t \left[ f(x^t_g) + \dotprod{\nabla F(x^t_g, Z_t)}{x^{t+1} - x^t_g}\right]
            \\&+
            \gamma_t  \dotprod{\underbrace{\nabla f(x^t_g) - \nabla F(x^t_g, Z_t)}_{\delta^t}}{x^{t+1} - x^t_g}
            \\&=
            (\beta_t-1)\gamma_t f(x_f^t)
            +
            \gamma_t \left[ f(x^t_g) + \dotprod{\nabla F(x^t_g, Z_t)}{x^{t+1} - x^t_g}\right]
            \\&+
            \gamma_t \left[ \dotprod{\delta^t}{x^{t} - x^t_g} + \dotprod{\delta^t}{d^t}\right]
            \\&\leq
            (\beta_t-1)\gamma_t f(x_f^t)
            +
            \gamma_t \left[f(x^t_g) + \dotprod{\nabla F(x^t_g, Z_t)}{x^{t+1} - x^t_g}\right]
            \\&+
            \gamma_t \left[ \dotprod{\delta^t}{x^{t} - x^t_g} + \|\delta^t\|_*\|d^t\|\right] .
        \end{split}
        \end{equation*}
        where $(a)$ is due to the convexity of the function $f$. Using this result and \eqref{eq:tmp_th1_1} we can obtain
        \begin{equation}
        \label{eq:tmp_th1_2}
        \begin{split}
            \beta_t \gamma_t f(x^{t+1}_f) 
            &\leq
            \beta_t \gamma_t \left[ f(x^t_g) + \dotprod{\nabla f(x^t_g)}{x^{t+1}_f - x^t_g}\right]
            +
            V(x^t, x^{t+1}) - \frac{\beta_t - L \gamma_t}{2 \beta_t}\norms{d^t}^2
            \\&\leq
            (\beta_t-1)\gamma_t f(x_f^t)
            +
            \gamma_t \left[ f(x^t_g) + \dotprod{\nabla F(x^t_g, Z_t)}{x^{t+1} - x^t_g}\right]
            \\&+
            \gamma_t \left[ \dotprod{\delta^t}{x^{t} - x^t_g} + \|\delta^t\|_*\|d^t\|\right] 
            +
            V(x^t, x^{t+1}) - \frac{\beta_t - L \gamma_t}{2 \beta_t}\norms{d^t}^2
        \end{split}
        \end{equation}
        Using Lemma 1 from \cite{lan2012optimal} with $\tilde{x} = x^t$ and $p(\xi) = \gamma_t \dotprod{\nabla F(x^t_g)}{\xi - x^t_g}$ and line 5 of Algorithm \ref{alg:MDGwb} we obtain that for all $x \in \cX$ it holds that
        \begin{equation*}
        \begin{split}
            \gamma_t \Big[ f(x^t_g) &+ \dotprod{\nabla F(x^t_g, Z_t)}{x^{t+1} - x^t_g}\Big] + V(x^t, x^{t+1})
            \\&\leq
            \gamma_t f(x^t_g) + \gamma_t \dotprod{\nabla F(x^t_g, Z_t)}{x - x^t_g} + V(x^t, x) - V(x^{t+1}, x)
            \\&=
            \gamma_t f(x^t_g) + \gamma_t \dotprod{\nabla f(x^t_g)}{x - x^t_g} 
            -
            \gamma_t \dotprod{\delta^t}{x - x^t_g}
            + 
            V(x^t, x) - V(x^{t+1}, x)
            \\&\leq
            \gamma_t f(x)
            -
            \gamma_t \dotprod{\delta^t}{x - x^t_g}
            + 
            V(x^t, x) - V(x^{t+1}, x) .
        \end{split}
        \end{equation*}
        Where the last inequality is due convexity of function $f$. Now we combine this result with \eqref{eq:tmp_th1_2} and use Fenchel-Young inequality \ref{axil:fenchel_young} with $\kappa = \eta_t$:
        \begin{equation}
        \label{eq:res_lemma_MDwb}
        \begin{split}
            \beta_t \gamma_t f(x^{t+1}_f) 
            &\leq
            (\beta_t-1)\gamma_t f(x_f^t)
            +
            \gamma_t \left[ f(x^t_g) + \dotprod{\nabla F(x^t_g, Z_t)}{x^{t+1} - x^t_g}\right]
            \\&+
            \gamma_t \left[ \dotprod{\delta^t}{x^{t} - x^t_g} + \|\delta^t\|_*\|d^t\|\right] 
            +
            V(x^t, x^{t+1}) - \frac{\beta_t - L \gamma_t}{2 \beta_t}\norms{d^t}^2
            \\&\leq
            (\beta_t-1)\gamma_t f(x_f^t) + \gamma_t f(x)
            +
            \gamma_t \dotprod{\delta^t}{x^{t} - x} 
            + 
            V(x^t, x) - V(x^{t+1}, x)
            \\&+
            \frac{\gamma_t^2}{2 \eta_t}\| \delta^t\|_*^2
            -
            \frac{1}{2}\left(1 - \eta_t - \frac{L \gamma_t}{\beta_t} \right)\norms{d^t}^2 .
        \end{split}
        \end{equation}
        Subtraction $\beta_t \gamma_t f(x)$ from both sides of \eqref{eq:res_lemma_MDwb} finishes the proof.
    \end{proof}
    Now we ready to prove Theorem \ref{theorem:MDGwb}.
    \begin{proof}[Proof of Theorem \ref{theorem:MDGwb}]
        Let $x^* \in \cX$ be and arbitrary solution to \eqref{eq:problem1}, i.e. $f(x^*) = f^*$. Then using Lemma \eqref{lemma:MDwb} with $x = x^*$ we obtain
        \begin{equation}
        \label{eq:tmp_th1_3}
        \begin{split}
            \beta_t \gamma_t (f(x^{t+1}_f) - f^*)
            &\leq
            (\beta_t-1)\gamma_t (f(x_f^t) - f^*)
            +
            \gamma_t \dotprod{\nabla f(x^t_g) - \nabla F(x^t_g, Z_t)}{x^{t} - x^*} 
            \\&+ 
            V(x^t, x^*) - V(x^{t+1}, x^*)
            +
            \frac{\gamma_t^2}{2 \eta_t}\| \nabla f(x^t_g) - \nabla F(x^t_g, Z_t) \|_*^2
            \\&-
            \frac{1}{2}\left(1 - \eta_t - \frac{L \gamma_t}{\beta_t} \right)\norms{x^{t+1} - x^t}^2 .
        \end{split}
        \end{equation}
        If we define $\delta^t := \nabla f(x^t_g) - \nabla F(x^t_g, Z_t)$ then, we obtain
        $\dotprod{\delta^t}{x^{t} - x^*} =
            \dotprod{\delta^t}{x^{t-\tau} - x^*}
            +
            \sum_{s = 1}^{\tau}\dotprod{\delta^t}{x^{t-s+1} - x^{t-s}}.
        $
        Consider term of the from $\dotprod{\delta^t}{x^{t-\tau} - x^*}$. Taking the full mathematical expectation $\expect{\cdot}$, using tower property and Cauchy-Schwarz inequality \ref{axil:cauchy_schwarz}  we obtain 
        $
            \expect{\dotprod{\delta^t}{x^{t-\tau} - x^*}}
            =
            \expect{\dotprod{\EEb{t - \tau}{\delta^t}}{x^{t-\tau} - x^*}}
            \leq
            \expect{\|\EEb{t - \tau}{\delta^t}\|_*\|x^{t-\tau} - x^*\|},
        $
        where $\EEb{t - \tau}{\cdot}$ is conditional mathematical expectation with fixed randomness up to step $t - \tau$. Consider term of the form $\|\EEb{t - \tau}{\delta^t}\|_*$. 
        \begin{equation*}
        \begin{split}
            \|\EEb{t - \tau}{\delta^t}\|_*
            &=
            \left\| \sum\limits_{Z \in \cZ} \mathbb{P} \left\{ \delta^t = \nabla f(x^t_g) - \nabla F(x^t_g, Z) | \delta^{t - \tau} \right\} \delta^t \right\|_*
            \\&\overset{(a)}{=}
            \left\| \sum\limits_{Z \in \cZ} \left(\mathbb{P} \left\{ \delta^t = \nabla f(x^t_g) - \nabla F(x^t_g, Z) | \delta^{t - \tau} \right\} - \pi_Z \right) \delta^t \right\|_*
            \\&\leq
            \sum\limits_{Z \in \cZ} \left|\mathbb{P} \left\{ \delta^t = \nabla f(x^t_g) - \nabla F(x^t_g, Z) | \delta^{t - \tau} \right\} - \pi_Z \right| \left\| \delta^t \right\|_*
            \\&\overset{(b)}{\leq}
            \frac{1}{2^{\tau / \tau_{\text{mix}}} \pi_{\min}} \sum\limits_{Z \in M} \pi_Z \left\| \delta^t \right\|_*
            \overset{(c)}{\leq}
            \frac{1}{2^{\tau / \tau_{\text{mix}}} \pi_{\min}} \sigma .
        \end{split}
        \end{equation*}
        The $(a)$ transition is due to first part of Assumption \ref{as:noise}: $\EEb{\pi}{\nabla F(x, Z) - \nabla f(x)} = 0$, $(b)$ is due to second part Assumption \ref{as:noise} and definition $\pi_{\min} := \min_{Z \in \cZ} \pi_Z$ and $(c)$ follows from Assumption \ref{as:var}.

        Using the fact, that $\cX$ is compact, we can estimate term of the form $\dotprod{\delta^t}{x^{t-\tau} - x^*}$: 
        \begin{equation}
        \label{eq:th1_eps}
            \expect{\dotprod{\delta^t}{x^{t-\tau} - x^*}} \leq \frac{1}{2^{\tau / \tau_{\text{mix}}} \pi_{\min}} D \sigma.
        \end{equation}
        Using $s$ times Fenchel-Young \ref{axil:fenchel_young} inequality with $\kappa = \kappa_t$ and Assumption \ref{as:var}, we can obtain that 
        \begin{equation}
        \label{eq:th1_tau}
            \sum_{s = 1}^{\tau}\dotprod{\delta^t}{x^{t-s+1} - x^{t-s}}
            \leq
            \frac{\sigma^2 \tau }{2 \kappa_t}
            +
            \frac{1}{2} \sum_{s = 1}^{\tau} \kappa_t \|x^{t-s+1} - x^{t-s}\|^2.
        \end{equation}
        Combining \eqref{eq:th1_eps} and \eqref{eq:th1_tau}, and using Assumption \ref{as:var} we can estimate expectation of \eqref{eq:tmp_th1_3}:
        \begin{equation}
        \label{eq:tmp_th1_4}
        \begin{split}
            \beta_t \gamma_t \expect{f(x^{t+1}_f) - f^*}
            &\leq
            (\beta_t-1)\gamma_t \expect{f(x_f^t) - f^*}
            \\&+
            \frac{\gamma_t}{2^{\tau / \tau_{\text{mix}}} \pi_{\min}} D \sigma
            +
            \frac{\gamma_t \sigma^2 \tau}{2 \kappa_t}
            +
            \frac{\gamma_t}{2} \sum_{s = 1}^{\tau} \kappa_t \expect{\|x^{t-s+1} - x^{t-s}\|^2}.
            \\&+ 
            \expect{V(x^t, x^*)} - \expect{V(x^{t+1}, x^*)}
            +
            \frac{\gamma_t^2}{2 \eta_t} \sigma^2
            \\&-
            \frac{1}{2}\left(1 - \eta_t - \frac{L \gamma_t}{\beta_t} \right)\expect{\norms{x^{t+1} - x^t}^2} .
        \end{split}
        \end{equation}
        Using the fact, that $0 \leq (\beta_{t+1} - 1) \gamma_{t+1} \leq \beta_t \gamma_t$ and $\beta_\tau = 1$ and summing \eqref{eq:tmp_th1_4} from $t = \tau$ to $t = T-1$ we obtain
        \begin{equation}
        \label{eq:tmp_th1_5}
        \begin{split}
            (\beta_{T} - 1) \gamma_{T} \expect{f(x^{T}_f) - f^*}
            &\leq
            \sum\limits_{t=\tau}^T\frac{\gamma_t}{2^{\tau / \tau_{\text{mix}}} \pi_{\min}} D \sigma
            +
            \sum\limits_{t=\tau}^T \left( \frac{\gamma_t \tau}{2 \kappa_t} + \frac{\gamma_t^2}{2 \eta_t} \right) (T - \tau) \sigma^2
            \\&+ 
            \expect{V(x^\tau, x^*)} - \expect{V(x^{T}, x^*)}
            \\&+
            \sum\limits_{t = \tau}^{T-1} \frac{\gamma_t}{2} \sum_{s = 1}^{\tau} \kappa_t \expect{\|x^{t-s+1} - x^{t-s}\|^2}
            \\&-
            \sum\limits_{t = \tau}^{T-1} \frac{1}{2}\left(1 - \eta_t - \frac{L \gamma_t}{\beta_t} \right)\expect{\norms{x^{t+1} - x^t}^2} .
        \end{split}
        \end{equation}
        Using the fact, that for all $x, y \in \cX$ it holds that $0 \leq V(x, y) \leq D^2$, Fenchel-Young inequality \ref{axil:fenchel_young} with $\kappa = 1/(T-\tau)$, taking $\kappa_t = 1/(4 \gamma_t \tau^2), \eta_t = 1/4$ and using the fact that $\sum_{t = \tau}^{T-1} \sum_{s = 0}^t a_s \leq \tau \sum_{t=0}^{T-1} a_t$, we can estimate \eqref{eq:tmp_th1_5}:
        \begin{equation}
        \label{eq:tmp_th1_6}
        \begin{split}
            (\beta_{T} - 1) \gamma_{T} \expect{f(x^{T}_f) - f^*}
            &\leq
            \frac{3 D^2}{2}
            +
            \sum\limits_{t=\tau}^T \gamma_t^2 \left( \frac{T-\tau}{4^{\tau / \tau_{\text{mix}}}\pi_{\min}^2} + \frac{\tau^3}{4} \right) \sigma^2
            \\&-
            \frac{1}{2}\sum\limits_{t = \tau}^{T-1} \left( \frac{1}{2} - \frac{L \gamma_t}{\beta_t} \right)\expect{\|x^{t+1} - x^{t}\|^2}
            \\&+
            \frac{1}{8 \tau}\sum\limits_{t = 0}^{\tau-1}\expect{\norms{x^{t+1} - x^t}^2} .
        \end{split}
        \end{equation}
        Taking $\beta_t \geq 2 L \gamma_t$ and using the fact that for all $x, y \in \cX$ holds that $\|x-y\|^2 \leq 2 D^2$ we obtain
        \begin{equation*}
        \begin{split}
            (\beta_{T} - 1) \gamma_{T} \expect{f(x^{T}_f) - f^*}
            &\leq
            \left(\frac{3}{2} + \frac{1}{4} \right) D^2
            +
            \left( \frac{T-\tau}{4^{\tau / \tau_{\text{mix}}} \pi_{\min}^2} + \frac{\tau^3}{4} \right) \sigma^2 \sum\limits_{t=\tau}^T \gamma_t^2.
        \end{split}
        \end{equation*}
        Now we need to compare two term of the form $\frac{T-\tau}{4^{\tau / \tau_{\text{mix}}} \pi_{\min}^2}$ and $\frac{\tau^3}{4}$. If we consider $\tau = \tau_{\text{mix}} \tilde{\tau}$, then we obtain that $\tilde{\tau} \sim W(\sqrt{3}{T} / \tau_{\text{mix}})$, where $W$ is Lambert $W$ function \cite{lehtonen2016lambert} and it grows slower than the logarithm function. Therefore we can conclude, that $\tau = \tilde{\cO}(\tau_{\text{mix}})$ and get the result of the form 
        \begin{equation*}
            (\beta_{T} - 1) \gamma_{T} \expect{f(x^{T}_f) - f^*}
            \leq
            2 D^2
            +
            \frac{\tau_{\text{mix}}^3}{4} \sigma^2 \sum\limits_{t=\tau_{\text{mix}}}^T \gamma_t^2.
        \end{equation*}    
        This finishes the proof.
    \end{proof}
    \begin{proof}[Proof of Corollary \ref{corollary:MDGwb}]
        From Theorem \ref{theorem:MDGwb} we can obtain, that
        \begin{equation*}
            (\beta_{T} - 1) \gamma_{T} \expect{f(x^{T}_f) - f^*}
            = \cO \left(
            D^2
            +
            \tau_{\text{mix}}^3 \sigma^2 \sum\limits_{t=\tau_{\text{mix}}}^T \gamma_t^2 \right) .
        \end{equation*}
        Choosing $\beta_t$ and $\gamma_t$ as 
        \begin{equation*}
            \beta_t := \max\left\{ \frac{t - \tau_{\text{mix}}}{2} + 1 ~;~ 1 \right\} 
            ~\text{ and }~
            \gamma_t := \max\left\{ \frac{t - \tau_{\text{mix}}}{2} + 1 ~;~ 1 \right\} \cdot 
            \gamma_*,
        \end{equation*}
        we obtain:
        \begin{equation*}
            \expect{f(x^{T}_f) - f^*}
            = \cO \left(
            \frac{D^2}{\gamma_* (T-\tau_{\text{mix}})^2}
            +
            \frac{\tau_{\text{mix}}^3 \sigma^2}{(T-\tau_{\text{mix}})^2 \gamma_*} \sum\limits_{t=\tau_{\text{mix}}}^T \gamma_t^2 \right) .
        \end{equation*}
        Using the fact that $\sum_{t=0}^N t^2 \leq \int_0^N x^2 dx = N^3/3$, we obtain
        \begin{equation*}
            \expect{f(x^{T}_f) - f^*}
            = \cO \left(
            \frac{D^2}{\gamma_* (T-\tau_{\text{mix}})^2}
            +
            \tau_{\text{mix}}^3 \sigma^2 \gamma_*(T-\tau_{\text{mix}}) \right) .
        \end{equation*}
        Choosing $\gamma_*$ as 
        $$\gamma_* := \min\left\{ \frac{1}{2L} ~;~ \frac{D}{(T - \tau_{\text{mix}})^{3/2} \sigma \tau_{\text{mix}}^{3/2}} \right\},$$
        we obtain
        \begin{equation*}
            \expect{f(x^{T}_f) - f^*}
            = \cO \left(
            \frac{2 L D^2}{(T-\tau_{\text{mix}})^2}
            +
            \frac{D \tau_{\text{mix}}^{3/2} \sigma}{\sqrt{T-\tau_{\text{mix}}}} \right) .
        \end{equation*}
        This finishes the proof.
    \end{proof}
\section{Proof of Lemmas \ref{lem:xuivjopeentertainment} and \ref{lem:expect_bound_grad}} 
\label{appendix:two_lemmas}
\begin{proof}[Proof of Lemma \ref{lem:xuivjopeentertainment}]
    Let $B_{\|\cdot\|}^1 := \{x \in \R^d : \|x\| = 1\}$, then for $\gamma > 0$, witch will be chosen later we define the sequence $\{u\}_{t=1}^{\infty}$ as follows:
    \begin{equation*}
        u^1 = 0,~~~~u^{t+1}:=\underset{u\in B_{\|\cdot\|}^1}{\argmin}\Big\{ \hat{V}(u, u^t) + \langle \gamma\xi^t, u\rangle\Big\},
    \end{equation*}
    where $\hat{V}(\cdot, \cdot)$ - is an arbitrary Bregman divergence. From \cite{juditsky2011solving} Corollary 2 we know that:
    \begin{equation*}
        \underset{u\in B_{\|\cdot\|}^1}{\max}\Bigg\{\bigg\langle \sum\limits_{t = 1}^N\gamma\xi^t, u\bigg\rangle\Bigg\} \leq \underbrace{\underset{u\in B_{\|\cdot\|}^1}{\max}\{\hat{V}(u, u^1)\}}_{\circledOne} + \underbrace{\frac{1}{2}\sum\limits_{t = 1}^N\|\gamma\xi^t\|^2_*}_{\circledTwo} + \underbrace{\sum\limits_{t = 1}^N\langle \gamma\xi^t, u^t\rangle}_{\circledThree}.
    \end{equation*}
    The first term $\circledOne$ is bounded the definition of $R^2:$  $R^2 := \max_{u \in B_{\|\cdot\|}^1}\{\hat{V}(0, u)\}$. From conditions of the Lemma we have: $\circledTwo \leq \gamma / 2 N \sigma^2$. Then by definition of the conjugate norm we have:
    \begin{equation*}
        \left\|\sum\limits_{t = 1}^N\xi_t\right\|_* \leq \frac{R^2}{\gamma} + \frac{\gamma N\sigma^2}{2} + \sum\limits_{t=1}^N\langle\xi^t, u^t\rangle.
    \end{equation*}
    Using the convexity of the squared norm we have:
    \begin{equation*}
        \begin{split}
            \left\|\sum\limits_{t = 1}^N\xi^t\right\|^2_* &\leq \frac{3R^4}{\gamma^2} + \frac{3\gamma^2B^2\sigma^4}{4} + 3\Bigg(\sum\limits_{t=1}^N\langle\xi^t, u^t\rangle\Bigg)^2\\
            &=
            \frac{3R^4}{\gamma^2} + \frac{3\gamma^2B^2\sigma^4}{4} + 3\sum\limits_{t=1}^N\langle\xi^t, u^t\rangle^2 + 3\sum\limits_{t \neq j}^N(\langle\xi^t, u^t\rangle\cdot\langle\xi^j, u^j\rangle).
        \end{split}
    \end{equation*}
    Consider the last term. Taking the full mathematical expectation one can obtain:
    \begin{equation*}
        \sum\limits_{t\neq j}\E[\langle\xi^t, u^t\rangle\cdot\langle\xi^j, u^j\rangle] = 2\sum\limits_{0\leq t -j \leq \tau}\E[\langle\xi^t, u^t\rangle\cdot\langle\xi^j, u^j\rangle] + 2\sum\limits_{t - j > \tau}\E[\langle\xi^t, u^t\rangle\cdot\langle\xi^j, u^j\rangle].
    \end{equation*}
    Now using the fact that $u\in B_{\|\cdot\|}^1$ and Cauchy-Schwarz inequality, we have:
    \begin{equation*}
        \sum\limits_{t\neq j}\E[\langle\xi^t, u^t\rangle\cdot\langle\xi^j, u^j\rangle] \leq 2\tau N\sigma^2 + 2\sum\limits_{t - j > \tau}\underbrace{\E[\langle\xi^t, u^t\rangle\cdot\langle\xi^j, u^j\rangle]}_{\circledFour}.
    \end{equation*}
    Consider $\circledFour$. Using tower property one can obtain:
    \begin{equation*}
        \E[\langle\xi^t, u^t\rangle\cdot\langle\xi^j, u^j\rangle] = \E[\E[\langle \xi^t, u^t\rangle|\xi^0, \ldots, \xi^{t - \tau}]\cdot \langle \xi^t, u^t\rangle].
    \end{equation*}
    Hence, we have:
    \begin{equation*}
    \begin{split}
        \E[\langle \xi^t, u^t\rangle|\xi^0, \ldots, \xi^{t - \tau}] &= \E_{t - \tau}[\langle \xi^t, u^t\rangle] = \E_{t - \tau}[\langle \xi^t, u^{t-\tau}\rangle] + \E_{t - \tau}[\langle \xi^t, u^t - u^{t-\tau}\rangle] \\
        &= \E_{t - \tau}[\langle \xi^t, u^{t-\tau}\rangle] + \sum\limits_{s = 1}^{\tau}\E_{t - \tau}[\langle \xi^t, u^{t-s+1} - u^{t-s}\rangle] \\
        &= \E_{t - \tau}[\langle \xi^t, u^{t-\tau}\rangle] + \sum\limits_{s = 1}^{\tau}\E_{t - \tau}[\langle \xi^t, \hat{P}_{u^{t-s}}(\gamma\xi^{t-s}) - \hat{P}_{u^{t-s}}(0)\rangle].
    \end{split}
    \end{equation*}
    Define $\mathcal{M}$ as state space of the Markov chain $\{\xi_t\}_{t=0}^\infty$ and consider the first summand:
    \begin{equation*}
    \begin{split}
        \E_{t - \tau}[\langle \xi^t, u^{t-\tau}&\rangle] \leq \|\E[\xi^t|\xi^{t-\tau},\ldots, \xi^{t-1}]\|_*\cdot\|u^{t - \tau}\|\\
        &\overset{(a)}{=}\|\E[\xi^t|\xi^{t-\tau},\ldots, \xi^{t-1}]\|_* \overset{(b)}{=} \left\| \|\sum\limits_{\nu\in \mathcal{M}}(\mathbb{P}\{\xi^t = \nu|\xi^{t-\tau}\} - \pi_{\nu})\nu \right\|_*\\
        &\leq \sum\limits_{\nu\in \mathcal{M}}|\mathbb{P}\{\xi^t = \nu|\xi^{t-\tau}\} - \pi_\nu|\cdot\|\nu\|_* \overset{(c)}{\leq} (1/2)^{\tau / \tau_{\text{mix}}} \sum\limits_{\nu\in \mathcal{M}} \pi_\nu\|\nu\|_* \leq (1/2)^{\tau / \tau_{\text{mix}}} \sigma,
    \end{split}
    \end{equation*}
    where $(a)$ follows from $u\in B_{\|\cdot\|}^1$, $(b)$ derives from $\E_\pi[\xi_t] = 0$ and $(c)$ ensues from the Assumption \ref{as:noise}.
    Consider the second summand:
    \begin{equation*}
        \E_{t - \tau}[\langle \xi^t, P_{u^{t-s}}(\gamma\xi^{t-s}) - P_{u^{t-s}}(0)\rangle] \overset{(a)}{\leq} \sigma \|P_{u^{t-s}}(\gamma\xi^{t-s}) - P_{u^{t-s}}(0)\| \overset{(b)}{\leq} \sigma \|\gamma\xi^{t-s} - 0\|_* \leq \gamma \sigma^2,
    \end{equation*}
    where $(a)$ follows from Cauchy-Schwarz inequality \ref{axil:cauchy_schwarz} and $(b)$ follows from Lemma \ref{nemirovski:lem3}.\\
    Combining this considerations we have:
    \begin{equation*}
        \E[\langle\xi^t, u^t\rangle\cdot\langle\xi^j, u^j\rangle] \leq (2^{-\tau / \tau_{\text{mix}}}\sigma + \tau\gamma\sigma^2)\E[\langle\xi_j, u_j\rangle]
    \end{equation*}
    Then for $j \leq \tau$ one can obtain: $$\E[\langle\xi^t, u^t\rangle\cdot\langle\xi^j, u^j\rangle] \leq (2^{-\tau / \tau_{\text{mix}}} +\tau\gamma)\sigma^2,$$
    and for $j>\tau$: $$\E[\langle\xi^j, u^j\rangle] \leq \varepsilon\sigma +\tau\gamma\sigma^2.$$
    Thus, we have:
    \begin{equation*}
        \E[\langle\xi_t, u_t\rangle\cdot\langle\xi_j, u_j\rangle] \leq 
        \begin{cases}
            \sigma^2, &~~\text{if } j\leq t\leq j+\tau\\
            (2^{-\tau / \tau_{\text{mix}}} + \tau\gamma)\sigma^2, &~~ \text{if } j+\tau<t \text{ and } j \leq \tau \\
            2(4^{-\tau / \tau_{\text{mix}}}+\tau^2\gamma^2\sigma^2)\sigma^2, &~~ \text{if } j + \tau < t \text{ and } j > \tau. 
        \end{cases}
    \end{equation*}
    Consequently, we obtain:
    \begin{equation*}
        \begin{split}
            2\sum\limits_{t - j > \tau}\E[\langle\xi^t, u^t\rangle\cdot\langle\xi^j, u^j\rangle] &\leq 2\tau N\sigma^2 + 2 \tau N(2^{-\tau / \tau_{\text{mix}}}+\tau\gamma)\sigma^2 + 2N^2(4^{- \tau / \tau_{\text{mix}}} + \tau^2\gamma^2\sigma^2)\sigma^2\\
            &\leq 4\tau N\sigma^2 + 2N^2(4^{- \tau / \tau_{\text{mix}}}+\tau^2\gamma^2\sigma^2)\sigma^2.
        \end{split}
    \end{equation*}
    Wrapping things up:
    \begin{equation*}
        \begin{split}
            \E\left\|\sum\limits_{t = 1}^N\xi^t\right\|^2_* \leq \frac{3R^4}{\gamma^2} + \frac{3\gamma^2N^2\sigma^4}{4} + 3(4\tau N\sigma^2 + 2N^2(4^{- \tau / \tau_{\text{mix}}}+\tau^2\gamma^2\sigma^2)\sigma^2).
        \end{split}
    \end{equation*}
    Taking $\gamma = \frac{R}{6\sqrt{N\tau}}$, we get:
    \begin{equation*}
        \begin{split}
        \E\left\|\sum\limits_{t = 1}^N\xi^t \right\|^2_* \leq 3R^2N\tau\sigma^2+\frac{3N\sigma^2R^2}{4\tau} &+ 3(5\tau N \sigma^2 + 4^{- \tau / \tau_{\text{mix}}} N^2\sigma^2 + R^2N\tau\sigma^2)\\ &\leq \Big[7R^2N\tau + 15N\tau+4^{- \tau / \tau_{\text{mix}}}N^2\Big]\sigma^2.
        \end{split}
    \end{equation*}
    Now, in case of $N \leq \tau_\text{mix}$ we take $\tau = 0$ and obtain:
    \begin{align*}
        \E\left\|\sum\limits_{t = 1}^N\xi^t \right\|^2_* \leq 4^{-\tau/\tau_\text{mix}}N^2\sigma^2 \leq N\tau_\text{mix}\sigma^2.
    \end{align*}
    And in case of $N > \tau_\text{mix}$ we again need to compare terms of the form $\tau$ and $4^{- \tau / \tau_{\text{mix}}}N$, similar to our discussion at the end of the proof of Theorem \ref{theorem:MDGwb} in Section \ref{appendix:MDGwb}, we can conclude that $\tau = \tau_{\text{mix}} \tilde{\cO}(1)$ and obtain the result of the form
    \begin{equation*}
        \begin{split}
        \E\left\|\sum\limits_{t = 1}^N\xi^t \right\|^2_* 
        &\leq \Big[7R^2N\tau_{\text{mix}} + 15N\tau_{\text{mix}}\Big]\sigma^2.
        \end{split}
    \end{equation*}
    Dividing both sides by $N^2$ concludes the proof.
\end{proof}
\begin{proof}[Proof of Lemma \ref{lem:expect_bound_grad}]
    To show that $\E_t[g^t] = \E_t[g^{t}_{\lfloor \log_2 M \rfloor}]$ we simply compute conditional expectation w.r.t. $J_t$:
\begin{equation}
\label{eq:tech:lem3}
    \begin{split}
        \E_t[g^t] &= \E_k\left[\E_{J_t}[g^t]\right] =
        \E_t[g^k_0] + \sum\limits_{i=1}^{\lfloor \log_2 M \rfloor} \mathbb{P}\{J_t = i\} \cdot 2^i \E_t[g^{t}_{i}  - g^{t}_{i-1}] \\
        &= \E_t[g^t_0] + \sum\limits_{i=1}^{\lfloor \log_2 M \rfloor} \E_t[g^{t}_{i}  - g^{t}_{i-1}] = \E_t[g^{t}_{\lfloor \log_2 M \rfloor}]\,.
    \end{split}
\end{equation}
We start with the proof of the first statement of \ref{lem:expect_bound_grad} by taking the conditional expectation for $J_t$:
\begin{align*}
    &\E_{t}[\| \nabla f(x^t_g) - g^t\|^2]
    \leq
    2\E_{t}[\| \nabla f(x^t_g) - g^t_0\|^2] + 2\E_{t}[\| g^t - g^t_0\|^2]
    \\
    &=
    2\E_{t}[\| \nabla f(x^t_g) - g^t_0\|^2] + 2 \sum\nolimits_{i=1}^{\lfloor \log_2 M \rfloor} \mathbb{P}\{J_t = i\} \cdot 4^i \E_t[\|g^{t}_{i}  - g^{t}_{i-1}\|^2] \\
    &=
    2\E_{t}[\| \nabla f(x^t_g) - g^t_0\|^2] + 2\sum\nolimits_{i=1}^{\lfloor \log_2 M \rfloor} 2^i \E_t[\|g^{t}_{i}  - g^{t}_{i-1}\|^2] \\
    &\leq
    2\E_{t}[\| \nabla f(x^t_g) - g^t_0\|^2] + 4\sum\nolimits_{i=1}^{\lfloor \log_2 M \rfloor} 2^i \left(\E_t[\|\nabla f(x^t_g)  - g^{t}_{i-1}\|^2 + \E_t[\|g^{t}_{i}  - \nabla f(x^t_g)\|^2] \right)\,.
\end{align*}
To bound $\E_{t}[\| \nabla f(x^t_g) - g^t_0\|^2]$, $\E_{t}[\|\nabla f(x^t_g)  - g^{t}_{i-1}\|^2$, $\E_{t}[\|g^{t}_{i}  - \nabla f(x^t_g)\|^2]$, we apply Lemma \ref{lem:xuivjopeentertainment} and get
\begin{align*}
    \E_{t}[\| \nabla f(x^t_g) - g^t\|^2]
    &\leq 2B^{-1}\sigma^{2} + 4B    ^{-1}\sum\nolimits_{i=1}^{\lfloor \log_2 M \rfloor} 2^i \cdot \frac{2(7R^2+15) \tau_{\text{mix}}}{2^{i}} \sigma^2 \\
    &\leq
    8(7R^2+15)B^{-1}\tau_{\text{mix}}\log_2 M \sigma^2 \,.
\end{align*}
To show the second part of the statement, we use \eqref{eq:tech:lem3} and get
\[
\| \nabla f(x^t_g) - \E_{t}[g^t]\|^2 = \| \nabla f(x^t) - \E_k[g^{t}_{\lfloor \log_2 M \rfloor}]\|^2\,.
\]
Applying Lemma \ref{lem:xuivjopeentertainment} and $2^{\lfloor \log_{2}M \rfloor} \geq M/2$ finishes the proof.
\end{proof}
\section{Proofs of results for \texttt{MAMD} with batching (Algorithm \ref{alg:MDG})}
\label{appendix:MDG}
    \begin{proof}[Proof of Theorem \ref{theorem:MDG}.]
        We begin by writing out result from Lemma \ref{lemma:MDwb} with $x = x^*$, $\eta_t = 1 - L\gamma_t / \beta_t \geq 0$ and replacing $\nabla F(x_g^t, Z_t)$ by $g^t$, because in Algorithm \ref{alg:MDG} we use $g^t$ as the gradient on the step $t$. Therefore we can obtain that
        \begin{equation}
        \label{eq:tmp_th2_1}
        \begin{split}
            \beta_t \gamma_t (f(x^{t+1}_f) - f^*)
            &\leq
            (\beta_t-1)\gamma_t (f(x_f^t) - f^*)
            +
            \gamma_t \dotprod{\nabla f(x^t_g) - g^t}{x^{t} - x^*} 
            \\&+ 
            V(x^t, x^*) - V(x^{t+1}, x^*)
            +
            \frac{\gamma_t^2 \beta_t}{2 (\beta_t - L \gamma_t)}\| \nabla f(x^t_g) - g^t \|_q^2 .
        \end{split}
        \end{equation}
        Using tower property and Cauchy-Schwarz inequality \ref{axil:cauchy_schwarz} we can obtain
        \begin{equation*}
        \begin{split}
            \expect{\dotprod{\nabla f(x^t_g) - g^t}{x^{t} - x^*}}
            &=
            \expect{\dotprod{\nabla f(x^t_g) - \EEb{t}{g^t}}{x^{t} - x^*}}
            \\&\leq
            \expect{\|\nabla f(x^t_g) - \EEb{t}{g^t} \|_q \| x^{t} - x^*\|_p}
        \end{split}
        \end{equation*}
        Using the first claim of the Lemma \ref{lem:expect_bound_grad}, the fact that $\forall x, y \in \cX$ it holds that $\|x - y\|_p \leq \sqrt{2} D$ and Fenchel-Young inequality \ref{axil:fenchel_young} with $\kappa = T$ we get
        \begin{equation}
        \label{eq:tmp_th2_2}
        \begin{split}
            \gamma_t \expect{\dotprod{\nabla f(x^t_g) - g^t}{x^{t} - x^*}}
            &\lesssim
            \gamma_t \sqrt{2 \tau_{\text{mix}} M^{-1} B^{-1} \sigma^2 D^2}
            \\&\lesssim
            T^{-1} D^2 
            +
            T \gamma_t^2 \tau_{\text{mix}} M^{-1} B^{-1} \sigma^2 .
        \end{split}
        \end{equation}
        Using tower property and the second claim of the Lemma \ref{lem:expect_bound_grad} we obtain that 
        \begin{equation}
        \label{eq:tmp_th2_3}
            \expect{\| \nabla f(x^t_g) - g^t \|_q^2}
            =
            \expect{\EEb{t}{\| \nabla f(x^t_g) - g^t \|_q^2}}
            \lesssim
            \tau_{\text{mix}} B^{-1} \log M \sigma^2 .
        \end{equation}
        Now using \eqref{eq:tmp_th2_2}, \eqref{eq:tmp_th2_3} and the fact, that $\beta_t - L\gamma_t \geq \beta_t/2$ we can estimate the expectation of \eqref{eq:tmp_th2_1}:
        \begin{equation}
        \label{eq:tmp_th2_4}
        \begin{split}
            \beta_t \gamma_t \expect{f(x^{t+1}_f) - f^*}
            &\lesssim
            (\beta_t-1)\gamma_t \expect{f(x_f^t) - f^*}
            +
            T^{-1} D^2 
            + 
            V(x^t, x^*) - V(x^{t+1}, x^*)
            \\&+
            T \gamma_t^2 \tau_{\text{mix}} M^{-1} B^{-1} \sigma^2 
            +
            \gamma_t^2 \tau_{\text{mix}} B^{-1} \log M \sigma^2 .
        \end{split}
        \end{equation}
        Using the fact, that $0 \leq (\beta_{t+1} - 1) \gamma_{t+1} \leq \beta_t \gamma_t$ and $\beta_0 = 1$, $0 \leq V(x, y) \leq D^2$ and summing \eqref{eq:tmp_th2_4} from $t = 0$ to $t = T-1$ we obtain
        \begin{equation*}
        \begin{split}
            (\beta_T - 1) \gamma_T \expect{f(x^{T}_f) - f^*}
            &\lesssim
            2 D^2 
            +
            \tau_{\text{mix}} B^{-1} \left( T M^{-1} + \log M \right) \sigma^2 \sum\limits_{t = 0}^{T-1}\gamma_t^2 .
        \end{split}
        \end{equation*}
        This finishes the proof .
    \end{proof}
    \begin{proof}[Proof of Corollary \ref{corollary:MDG}]
        As it was done in Corollary \ref{corollary:MDG} (see Appendix \ref{appendix:MDGwb})
        Choosing $\beta_t$ and $\gamma_t$ as 
        \begin{equation*}
            \beta_t := \frac{t}{2} + 1
            ~\text{ and }~
            \gamma_t := \left( \frac{t}{2} + 1 \right) \cdot 
            \gamma_*,
        \end{equation*}
        We obtain result of the from 
        \begin{equation*}
            \expect{f(x^{T}_f) - f^*}
            =\cO\left(
            \frac{D^2}{\gamma_* T^2} 
            +
            \tau_{\text{mix}} B^{-1} \left( T M^{-1} + \log M \right) \gamma_* \sigma^2 T \right).
        \end{equation*}
        If we choose $M = T / W(T) \sim T$ and $B = 1$, then we obtain result of the form 
        \begin{equation*}
            \expect{f(x^{T}_f) - f^*}
            =\widetilde{\cO}\left(
            \frac{D^2}{\gamma_* T^2} 
            +
            \tau_{\text{mix}} \gamma_* \sigma^2 T \right).
        \end{equation*}
        Choosing $\gamma_*$ as 
        $$\gamma_* := \min\left\{ \frac{1}{2L} ~;~ \frac{D}{T^{3/2} \sigma \tau_{\text{mix}}^{1/2}} \right\},$$
        gives us result of the form
        \begin{equation*}
            \expect{f(x^{T}_f) - f^*}
            =\widetilde{\cO}\left(
            \frac{L D^2}{T^2} 
            +
            \sqrt{\frac{\tau_{\text{mix}} D^2 \sigma^2}{T}} \right).
        \end{equation*}
        This finishes the proof .
    \end{proof}
\section{Proofs of results for \texttt{MMP} without batching (Algorithm \ref{alg:MPGwb})}
\label{appendix:th3}
\begin{proof}
    We start with using Lemma \ref{nemirovski:lem4} with $w = x^{t+\frac{1}{2}},~~ x = x^t,~~ r_+ = x^{t+1},~ \zeta = \gamma F(x^t, Z_t)~$ and $ \eta = \gamma F(x^{t+\frac{1}{2}}, Z_t)$, thus 
    \begin{equation}
    \label{eq:th3:1}
        \begin{split}
            \gamma\langle F(x^{t+\frac{1}{2}}, Z_t), x^{t+\frac{1}{2}} - u\rangle &+ V(x^{t+1}, u) - V(x^t, u) \\&\leq \frac{\gamma^2}{2}\|F(x^{t+\frac{1}{2}}, Z_t) - F(x^{t}, Z_t)\|^2_*-\frac{1}{2}\|x^{t+\frac{1}{2}} - x^t\|^2 \\&\leq \Bigg(\frac{\gamma^2\Tilde{L}^2}{2} - \frac{1}{2}\Bigg)\|x^{t+\frac{1}{2}} - x^t\|^2,
        \end{split}
    \end{equation}
    
    where the last inequality follows from Assumption \ref{as:lipvar_Z}. Using Assumption \ref{as:monotone}, one can obtain 
    \begin{equation*}
    \begin{split}
        \gamma \langle F(x^{t+\frac{1}{2}}, Z_t), x^{t+\frac{1}{2}} - u\rangle &\geq \gamma \langle F(u, Z_t), x^{t+\frac{1}{2}} - u\rangle \\&= \gamma \langle F(u), x^{t+\frac{1}{2}} - u\rangle + \gamma \langle F(u, Z_t) - F(u), x^{t+\frac{1}{2}} - u\rangle.
    \end{split}
    \end{equation*}
    Thus, \eqref{eq:th3:1} turns into
    \begin{equation}
        \label{eq:th3:3}
        \begin{split}
            \gamma\langle F(u), x^{t+\frac{1}{2}} - u\rangle &+ V(x^{t+1}, u) - V(x^t, u) \\&\leq \Bigg(\frac{\gamma^2\Tilde{L}^2}{2} - \frac{1}{2}\Bigg)\|x^{t+\frac{1}{2}} - x^t\|^2 + \gamma\langle F(u) - F(u, Z_t), x^{t+\frac{1}{2}} - u\rangle.
        \end{split}
    \end{equation}
    Consider the last term, for $t \geq \tau$ we have
    \begin{equation}
    \label{eq:th3:2}
        \begin{split}
            \langle F(u) &- F(u, Z_t), x^{t+\frac{1}{2}} - u\rangle \\&= \langle F(u) - F(u, Z_t), x^{t - \tau + \frac{1}{2}} - u\rangle + \langle F(u) - F(u, Z_t), x^{t+\frac{1}{2}} - x^{t - \tau + \frac{1}{2}}\rangle \\&\leq \langle F(u) - F(u, Z_t), x^{t - \tau + \frac{1}{2}} - u\rangle + \|F(u) - F(u, Z_t)\|_*\|x^{t+\frac{1}{2}} - x^{t- \tau +\frac{1}{2}}\|.
        \end{split}
    \end{equation}
    Consider $\|x^{t+\frac{1}{2}} - x^{t-\tau+\frac{1}{2}}\|$
    \begin{equation*}
        \begin{split}
           \|x^{t+\frac{1}{2}} - x^{t-\tau+\frac{1}{2}}\| &= \|x^{t+\frac{1}{2}} - x^t + x^t - x^{t-\tau+\frac{1}{2}}\| \leq \|x^{t+\frac{1}{2}} - x^t\| + \|x^t - x^{t-\tau+\frac{1}{2}}\| \\&= \|x^{t+\frac{1}{2}} - x^t\| + \|x^t - x^{t-\frac{1}{2}} + x^{t-\frac{1}{2}} - x^{t-\tau+\frac{1}{2}}\| = \|x^{t+\frac{1}{2}} - x^t\| \\&+ \|P_{x^{t-1}}(\gamma F(x^{t-\frac{1}{2}}, z^{t-1})) - P_{x^{t-1}}(\gamma F(x^{t-1}, z^{t-1})) + x^{t-\frac{1}{2}} - x^{t-\tau+\frac{1}{2}}\| \\&\leq \|x^{t+\frac{1}{2}} - x^t\| + \gamma\|F(x^{t-\frac{1}{2}}, z^{t-1}) - F(x^{t-1}, z^{t-1})\| + \|x^{t-\frac{1}{2}} - x^{t-\tau+\frac{1}{2}}\| \\&\leq \|x^{t+\frac{1}{2}} - x^t\| + \gamma \Tilde{L}\|x^{t-\frac{1}{2}} - x^{t-1}\| + \|x^{t-\frac{1}{2}} - x^{t-\tau+\frac{1}{2}}\|.
        \end{split}
    \end{equation*}
    Performing the recursion one can obtain 
    \begin{equation*}
        \|x^{t+\frac{1}{2}} - x^{t - \tau + \frac{1}{2}}\| \leq \sum\limits_{s = 0}^{\tau - 1}(1+\gamma \Tilde{L})\|x^{t - s + \frac{1}{2}} - x^{t - s}\|.
    \end{equation*}
    Thus, \eqref{eq:th3:2} turns into 
    \begin{equation}
        \begin{split}
            \langle F(u) - F(u, Z_t), x^{t+\frac{1}{2}} - 
            &x^{t - \tau + \frac{1}{2}}\rangle \leq \|F(u) - F(u, Z_t)\|_*\Bigg( \sum\limits_{s = 0}^{\tau - 1}(1+\gamma \Tilde{L})\|x^{t - s + \frac{1}{2}} - x^{t - s}\|\Bigg) \\&~~~\leq \frac{1}{\beta} \|F(u) - F(u, Z_t)\|_*^2 + \beta \Bigg( \sum\limits_{s = 0}^{\tau - 1}(1+\gamma \Tilde{L})\|x^{t - s + \frac{1}{2}} - x^{t - s}\|\Bigg)^2 \\&~~~\leq \frac{1}{\beta} \|F(u) - F(u, Z_t)\|_*^2 + \beta \tau (1+\gamma \Tilde{L})^2\sum\limits_{s = 0}^{\tau - 1}\|x^{t - s + \frac{1}{2}} - x^{t - s}\|^2.
        \end{split}
    \end{equation}
    Assume a notation $a^t := \E[\|x^{t+\frac{1}{2}} - x^t\|^2]$. Taking the expectation, from \eqref{eq:th3:3} we get 
    \begin{equation*}
    \begin{split}
        \gamma \E[\langle F(u), x^{t + \frac{1}{2}} - u\rangle] &\leq \expect{V(x^t, u)} - \expect{V(x^{t+1}, u)} + \frac{1}{2}(\gamma^2\Tilde{L}^2 - 1)a^t \\&+ \frac{\gamma}{\beta} \E\|F(u) - F(u, Z_t)\|_*^2 + \gamma\expect{\langle F(u) - F(u, Z_t), x^{t - \tau + \frac{1}{2}} - u\rangle} \\&+ \gamma \beta \tau (1+\gamma \Tilde{L})^2\sum\limits_{s = t-\tau+1}^{t - 1}a^s.
    \end{split}
    \end{equation*}
    Denoting $\xi^t = \|F(u) - F(u, Z_t)\|_*^2$ and $u^{t - \tau} = x^{t - \tau + \frac{1}{2}} - u$, we get 
    \begin{equation*}
        \begin{split}
            \E[\langle \xi^t, u^{t-\tau}\rangle] &= \E[\langle \E[\xi^t| \xi^0, \ldots, \xi^{t-\tau}], u^{t-\tau}\rangle] \leq \|\E[\xi^t|\xi^{t-\tau},\ldots, \xi^{t-1}]\|_*\cdot\|u^{t - \tau}\|\\
            &\leq\|\E[\xi^t|\xi^{t-\tau},\ldots, \xi^{t-1}]\|_* D {=} \left\| \sum\limits_{\nu\in \mathcal{M}}(\mathbb{P}\{\xi^t = \nu|\xi^{t-\tau}\} - \pi_{\nu})\nu \right\|_*D\\
            &\leq \sum\limits_{\nu\in \mathcal{M}}|\mathbb{P}\{\xi^t = \nu|\xi^{t-\tau}\} - \pi_\nu|\cdot\|\nu\|_*D \leq  \varepsilon D \sum\limits_{\nu\in \mathcal{M}} \pi_\nu\|\nu\|_*
        \end{split}
    \end{equation*}
    Therefore, we have 
    \begin{equation*}
    \begin{split}
        &\gamma \E[\langle F(u), x^{t + \frac{1}{2}} - u\rangle] \leq \expect{V(x^t, u)} - \expect{V(x^{t+1}, u)} + \frac{1}{2}(\gamma^2L^2 - 1)a^t \\&\quad+ \frac{\gamma}{\beta} \E\|F(u) - F(u, Z_t)\|_*^2 + \gamma \varepsilon \E_{\pi}[\|F(u) - F(u, \cdot)\|_*] D + \gamma \beta \tau (1+\gamma \Tilde{L})^2\sum\limits_{s = t-\tau+1}^{t - 1}a^s.
    \end{split}
    \end{equation*}
    Summing from $t = \tau$ to $T$ and dividing both sides by $T - \tau$, we obtain
    \begin{equation*}
        \begin{split}
            \gamma\E[\langle F(u), \frac{1}{T-\tau}&\sum\limits_{t = \tau}^Tx^{t+\frac{1}{2}} - u\rangle] \leq \frac{1}{T-\tau}V(x^\tau, u) + \frac{1}{2(T-\tau)}(\gamma^2L^2 - 1)\sum\limits_{t=\tau}^Ta^t \\&+\gamma(T - \tau)\varepsilon\E_{\pi}[\|F(u) - F(u, \cdot)\|_*]D + \frac{1}{T-\tau}\frac{\gamma}{\beta} \sum\limits_{t=\tau}^T\E\|F(u) - F(u, Z_t)\|_*^2 \\&+ \gamma \beta \tau (1+\gamma \Tilde{L})^2\sum\limits_{t = \tau}^T\sum\limits_{s = t-\tau+1}^{t - 1}a^s \\&\quad\quad\quad\quad\quad\quad~~\leq \frac{D^2}{T-\tau} + \frac{1}{T-\tau}\sum\limits_{t = \tau}^T\Bigg(\frac{\gamma^2\Tilde{L}^2}{2} - \frac{1}{2} + \gamma\tau^2\beta(1+\gamma \Tilde{L})^2\Bigg)a^t \\&+ 
            \frac{\gamma \beta \tau (1 + \gamma \Tilde{L})^2}{T-\tau} \sum\limits_{t = 0}^{\tau-1} a^t \\&+
            (T - \tau)\gamma\varepsilon\E_{\pi}[\|F(u) - F(u, \cdot)\|_*]D + \frac{1}{T-\tau}\frac{\gamma}{\beta} \sum\limits_{t=\tau}^T\E\|F(u) - F(u, Z_t)\|_*^2.
        \end{split}
    \end{equation*}
    Assume a notation $\frac{1}{T - \tau}\sum\limits_{t = \tau}^{T}x^{t+\frac{1}{2}} = \Hat{x}^{t+\frac{1}{2}}$. Choosing $\gamma\leq \frac{1}{2\Tilde{L}}$ and $\beta = \frac{1}{6\tau^2\gamma}$, we get 
    \begin{equation*}
        \begin{split}
            \E[\langle F(u), \Hat{x}^{t+\frac{1}{2}} - u\rangle] &\leq \frac{D^2}{\gamma (T - \tau)} + (T - \tau)\varepsilon\E_{\pi}[\|F(u) - F(u, \cdot)\|_*]D \\&\quad+ \frac{6\tau^2\gamma}{T - \tau}\sum\limits_{t=\tau}^{T}\E\|F(u) - F(u, Z_t)\|_*^2.
        \end{split}
    \end{equation*}
    For $t \geq \tau_{\text{mix}}$ we handle the last term
    \begin{equation*}
        \E[\xi^t] = \sum\limits_{Z\in\mathcal{Z}}\mathbb{P}\{\xi^t = Z\}Z \leq \sum\limits_{Z\in\mathcal{Z}}|\mathbb{P}\{\xi^t = Z\} - \pi_Z|Z + \sum\limits_{Z\in\mathcal{Z}}\pi_ZZ \leq (1+\varepsilon)\E_\pi[\xi^t].
    \end{equation*}
    Thus, we obtain 
    \begin{equation*}
        \begin{split}
            \E[\langle F(u), \Hat{x}^{t+\frac{1}{2}} - u\rangle] &\leq \frac{2D^2}{\gamma (T - \tau)} + \gamma\varepsilon^2 (T - \tau)^3\E_{\pi}[\|F(u) - F(u, \cdot)\|^2_*] \\&+ 6\tau^2\gamma(1+\varepsilon)\E_\pi[\|F(u) - F(u, \cdot)\|_*^2].
        \end{split}
    \end{equation*}
    Recall that $\varepsilon = (1/2)^{\frac{\tau}{\tau_{\text{mix}}}}$, one can get 
    \begin{equation*}
        \begin{split}
            \E[\langle F(u), \Hat{x}^{t+\frac{1}{2}} - u\rangle] &\leq \frac{2D^2}{\gamma (T - \tau)} + \gamma(12 \tau^2 +  \frac{(T-\tau)^{3}}{4^{\tau/\tau_{\text{mix}}}})\E_{\pi}[\|F(u) - F(u, \cdot)\|^2_*].
        \end{split}
    \end{equation*}
    Now we need to compare two term of the form $\frac{(T-\tau)^3}{4^{\tau / \tau_{\text{mix}}}}$ and $12\tau^2$. If we consider $\tau = \tau_{\text{mix}} \tilde{\tau}$, then we obtain that $\tilde{\tau} \sim W({T^{3/2}} / 2\sqrt{3}\tau_{\text{mix}})$, where $W$ is Lambert $W$ function \cite{lehtonen2016lambert} and it grows slower than the logarithm function. Therefore we can conclude, that $\tau = \tilde{\cO}(\tau_{\text{mix}})$ and get the result of the form
    \begin{equation*}
        \begin{split}
            \E[\langle F(u), \Hat{x}^{t+\frac{1}{2}} - u\rangle] &\leq \frac{2D^2}{\gamma (T - \tau)} + 12\gamma \tau_{\text{mix}}^2 \E_{\pi}[\|F(u) - F(u, \cdot)\|^2_*].
        \end{split}
    \end{equation*}
    Taking $\underset{u}{\max}$ from both sides concludes the proof.
\end{proof}
\begin{proof}[Proof of Corollary \ref{cor:th3}]
    From Theorem \ref{theorem:MPGwb} we can obtain, that
        \begin{equation*}
            \underset{u}{\max}\E[\langle F(u), \Hat{x}^{t+\frac{1}{2}} - u\rangle] \leq \frac{2D^2}{\gamma (T - \tau_{\text{mix}})} + 12\gamma \tau_{\text{mix}}^2\sigma^2.
        \end{equation*}
        Choosing $\gamma$ as 
        \begin{equation*}
            \gamma := \min\left\{ \frac{1}{2\Tilde{L}} ~;~ \frac{D}{(T - \tau_{\text{mix}})^{1/2} \sigma \tau_{\text{mix}}} \right\},
        \end{equation*}
        we obtain
        \begin{equation*}
            \underset{u}{\max}\E[\langle F(u), \Hat{x}^{t+\frac{1}{2}} - u\rangle]
            = \cO \left(
            \frac{2 \Tilde{L} D^2}{T-\tau_{\text{mix}}}
            +
            \frac{D \tau_{\text{mix}} \sigma}{\sqrt{T-\tau_{\text{mix}}}} \right) .
        \end{equation*}
        This finishes the proof.
\end{proof}
\section{Proofs of results for \texttt{MMP} with batching (Algorithm \ref{alg:MPG})}
    \label{appendix:th4}
    We start with result from \cite{juditsky2011solving}. 
    \begin{lemma}[Theorem 2 from \cite{juditsky2011solving}]
    \label{lemma:VI}
        Let the conditions of Theorem \ref{theorem:MPG} be satisfied. Let us define 
        $$\delta^t := F(x^{t+\frac{1}{2}}) - g^t.$$
        For $x^t$ belonging to the trajectory $\{x^0, x^\frac{1}{2},\ldots, x^T, x^{T+\frac{1}{2}}\}$ of the Algorithm \ref{alg:MPG}, let $$\epsilon^{1/2}_t = \|F(x^t) - g^{t+\frac{1}{2}}\|^2_q,\quad \epsilon^1_t = \|F(x^{t+\frac{1}{2}}) - g^t\|^2_q.$$
        And define auxiliary sequence $\{y^t\}_{t=0}^\infty$ as 
        $$y^{t} = P_{y^{t-1}}(\delta^t), ~~ y^0 = x^0.$$
        Then for all $T \geq 0$ and for all $u \in \cX$, it holds that 
        \begin{equation*}
        \begin{split}
            \gamma \dotprod{F(u)}{x^{t+1/2} - u} 
            &\leq 
            2 D^2
            +
            \frac{3\gamma^2}{2} \sum\limits_{t = 0}^{T-1}\left(2 \epsilon_t^{1/2} + 3 \epsilon_t^1 \right)
            +
            \gamma \sum\limits_{t = 0}^{T-1}\dotprod{\delta^t}{x^{t+1/2} - y^{t-1}} .
        \end{split}
        \end{equation*}
    \end{lemma}

    Now we are ready to prove Theorem \ref{theorem:MPG}.
    \begin{proof}[Proof of Theorem \ref{theorem:MPG}]
        Firstly let us take full mathematical expectation of the result of Lemma \ref{lemma:VI} and deal with noise terms separately. From Lemmas \ref{lem:xuivjopeentertainment} and \ref{lem:expect_bound_grad} we obtain that for all $t \geq 0$ it holds that
        \begin{equation*}
            \expect{\epsilon_t^{1/2}}
            \lesssim
            \tau_{\text{mix}} B^{-1} \sigma^2 
            ~\text{ and }~
            \expect{\epsilon^1_t}
            \lesssim
            \tau_{\text{mix}} B^{-1} \log(M) \sigma^2 .
        \end{equation*}
        From tower property, Cauchy-Schwarz inequality \ref{axil:cauchy_schwarz} and Lemma \ref{lem:expect_bound_grad} we obtain that 
        \begin{equation*}
        \begin{split}
            \expect{\dotprod{\delta^t}{x^{t+1/2} - y^{t-1}}}
            &\leq
            \left\| \EEb{t}{\delta^t} \right\|_q \|x^{t+1/2} - y^{t-1} \|_p
            \lesssim
            \sqrt{\tau_{\text{mix}} B^{-1}M^{-1} \sigma^2 D^2}. 
        \end{split}
        \end{equation*}
        Using notation $\widehat{x}^{T} := \frac{1}{T} \sum_{t=0}^{T-1} x^{t+1/2}$ we obtain that 
        \begin{equation*}
            \expect{\text{Err}_{\text{VI}}(\widehat{x}^{T})} 
            = 
            \widetilde{\cO} \left( 
                \frac{D^2}{\gamma T}
                +
                \gamma \tau_{\text{mix}} B^{-1} \left( 1 + \log(M) \right) \sigma^2
                +
                \sqrt{\tau_{\text{mix}} B^{-1}M^{-1} \sigma^2 D^2}
            \right).
        \end{equation*}
        Using Fenchel-Young inequality \ref{axil:fenchel_young} wuth $\kappa = \gamma T$ and the fact that $\log(M) \geq 1$ we obtain
        \begin{equation*}
            \expect{\text{Err}_{\text{VI}}(\widehat{x}^{T})} 
            = 
            \widetilde{\cO} \left( 
                \frac{D^2}{\gamma T}
                +
                \gamma \tau_{\text{mix}} B^{-1} \left( TM^{-1} + \log(M) \right) \sigma^2
            \right).
        \end{equation*}
        This finishes the proof.
    \end{proof}
\begin{proof}[Proof of Corollary \ref{cor:th4}]
    From Theorem \ref{theorem:MPGwb} we can obtain, that
        \begin{equation*}
           \expect{\text{Err}_{\text{VI}}(\widehat{x}^{T})} 
            = 
            \widetilde{\cO} \left( 
                \frac{D^2}{\gamma T}
                +
                \gamma \tau_{\text{mix}} B^{-1} \left( TM^{-1} + \log(M) \right) \sigma^2
            \right).
        \end{equation*}
        Choosing $\gamma$ as 
        \begin{equation*}
        \gamma := \min\left\{ \frac{1}{2L} ~;~ \frac{D}{T^{1/2} \sigma \tau_{\text{mix}}^{1/2} }\right\}
        ,~
        M =  T 
        ~\text{ and }~
        B = 1,
        \end{equation*}
        we obtain
        \begin{equation*}
            \expect{\text{Err}_{\text{VI}}(\widehat{x}^{T})} 
            = 
            \widetilde{\cO} \left( 
                \frac{LD^2}{T}
                +
                \frac{\tau_{\text{mix}} D \sigma}{\sqrt{T}}
            \right).
        \end{equation*}
        This finishes the proof.
\end{proof}
\section{Lower bounds}
\label{appendix:lover_bounds}

Firstly, we define the methods for which the lower bounds will be constructed. We say an iterative algorithm for solving \eqref{eq:problem1} is a \textit{stochastic first-order method} if it accesses the information of the function $f(x)$ through a \textit{stochastic first-order oracle}, denoted by $\mO:\cX\times\cZ\to\R^d$. For arbitrary point $(x, Z) \in \cX\times\cZ$, the oracle $\mO(x, Z)$ returns $\nabla F(x, Z)$.

Given an initial point $x^0$, a first-order method $\mM$ for solving problem \ref{eq:problem1}, at the $t$-th iteration, calls the oracle in the point $(x^t, Z_t)$ to collect the oracle information $\mO(x^t, Z_t)$ and then obtains a new point $x^{t+1}$. The complete method $\mM$ can be described by the sequence of points $\{x^t\}_{t=0}^{\infty}$ such that for all $t \geq 0$
\begin{equation*}
    x^{t+1} \in \text{Span}\left\{x^0, \dots, x^t, \mO(x^0, Z_0), \dots, \mO(x^t, Z_t) \right\} \cap \cX.
\end{equation*}
It is for algorithms of this class that we will prove lower bouns. For VI problems \eqref{eq:problem2}, class of the stochastic first-order methods is introduced similarly.
\begin{proof}[Proof of Proposition \ref{proposition:1}]
    Let $F^{\text{deter}}(x^1)$ -- is the function on which the lower bound for the deterministic problem from \cite{nesterov2013introductory} is reached, $F^{\text{stoch}}(x^2, Z)$ -- is the corresponding function for the Markov stochastic problem from \cite{duchi2012ergodic}. We define $x := (x^1, x^2)^T$ and $F(x, Z) = F^{\text{deter}}(x^1) + F^{\text{stoch}}(x^2, Z)$. Then we obtain that
    \begin{equation*}
        \nabla F(x, Z) = \left( \nabla_{x^1} F^{\text{deter}}(x^1), \nabla_{x^2} F^{\text{stoch}}(x^2, Z) \right)^T.
    \end{equation*}
    Since the gradient contains both terms of the form $\nabla_{x^1} F^{\text{deter}}(x^1)$ and $\nabla_{x^2} F^{\text{stoch}}(x^2, Z)$ separately, the convergence of any first-order method using Markov stochasticity will be of order $\widetilde{\cO}\left(\max\{ \Omega_{\text{deter}} ~;~ \Omega_{\text{stoch}} \}\right)$. This gives us the result of the Proposition \ref{proposition:1}. 
\end{proof}
\begin{proof}[Proof of Proposition \ref{proposition:2}]
    In case of variational inequalities problem \eqref{eq:problem2}, we can consider saddle point problems, as in the lower bound from \cite{ouyang2021lower} (it will be $F^{\text{deter}}(x^1)$), and a minimisation problem, as in the lower bound estimator from [2] (it will be $F^{\text{stoch}}(x^2, Z)$). Then Proposition \ref{proposition:2} is proved similarly to Proposition \ref{proposition:1}.
\end{proof}
\end{document}